\documentclass[11pt,english]{smfart}

\usepackage[T1]{fontenc}
\usepackage[english]{babel}

\usepackage{amssymb,url,xspace,smfthm}
\usepackage{amsfonts}
\usepackage{graphicx}

\newcommand{\BibTeX}{{\scshape Bib}\kern-.08em\TeX}
\newcommand{\T}{\S\kern .15em\relax }
\newcommand{\AMS}{$\mathcal{A}$\kern-.1667em\lower.5ex\hbox
        {$\mathcal{M}$}\kern-.125em$\mathcal{S}$}

\newtheorem{lemma}{Lemma}
\newtheorem{remark}{Remark}
\newtheorem{theorem}{Theorem}
\newtheorem{definition}{Definition}
\newtheorem{corollary}{Corrolary}
\title[Weak Solutions to an Euler Alignment System]{Weak Solutions \\ to an Euler Alignment System \\ with Singular Interactions \\ in a Bounded Domain}
\date {\today}
\author{Amoolya Tirumalai}

\address{Department of Electrical and Computer Engineering \& Institute for Systems Research, University of Maryland, 8223 Paint Branch Dr, College Park, MD 20740, USA. }
\email{ast256@umd.edu}

\author{Christos N. Mavridis}

\address{Department of Electrical and Computer Engineering \& Institute for Systems Research, University of Maryland, 8223 Paint Branch Dr, College Park, MD 20740, USA. }
\email{mavridis@umd.edu}

\author{John S. Baras}

\address{Department of Electrical and Computer Engineering \& Institute for Systems Research, University of Maryland, 8223 Paint Branch Dr, College Park, MD 20740, USA. }
\email{baras@umd.edu}


\begin{document}
\def\smfbyname{}

\begin{abstract}
Euler alignment systems appear as hydrodynamic limits of interacting self-propelled particle systems such as the (generalized) Cucker-Smale model. In this work, we study weak solutions to an Euler alignment system on smooth, bounded, connected domains. This particular Euler alignment system includes 
singular alignment, attraction, and repulsion interaction kernels which correspond to a Yukawa potential. We also include a confinement potential and self-propulsion. We embed the problem into an abstract Euler system to conclude that infinitely many weak solutions exist. We further show that we can construct solutions satisfying bounds on an energy quantity, and that the solutions satisfy a weak-strong uniqueness principle. Finally, we present an addition of leader-agents governed by controlled ODEs, and modification of the interactions to be Bessel potentials of fractional order $s > 2$.
\end{abstract}
\thanks{This work was supported in part by the US
Defense Advanced Research Projects Agency (DARPA) under Agreement No. HR00111990027, and by the US Office of Naval Research (ONR) Grant No. N00014-17-1-2622. The work herein expresses the views of the authors.}
\maketitle
\tableofcontents

\section{Introduction}
In the last two decades, considerable interest has been shown in models of collective behavior 
across a variety of scientific fields, including traffic, ecology, and robotics. A particular area of this interest is in using insights gained from biological
flocking and swarming to construct controls for large teams of UAVs \cite{degond2013hierarchy,lopez2012behavioural,
	ballerini2008empirical, sepulchre2007stabilization,vicsek2008universal, mavridis2020detection, mavridis2020learning, mavridis2022learning, tirumalai2022robust, tirumalai2023approximate}. 

There are a number of well-studied models of collective motion, including the Boids model, the Cucker-Smale model, and the Vicsek model \cite{reynolds1987flocks, cucker2007emergent, vicsek1995novel}. These three models can be seen as representatives of efforts by computer scientists, mathematicians/ control 
theorists, and physicists, respectively, to describe the same phenomenon .

Analysis of large systems of agents presents many difficulties, and passage to continuum limits alleviates some of them while introducing others. There is substantial interest in kinetic and hydrodynamic (continuum) limits of interacting particle systems \cite{albi2019vehicular, bellomo2017quest, figalli2018rigorous, 
	carrillo2017review, carrillo2010particle}. 

We are interested in formal hydrodynamic limits of particle dynamics subject to velocity alignment, attraction, repulsion, self-propulsion, and containment forces in bounded, smooth regions. Similar particle systems arise in design of controls for autonomous vehicles or UAVs using virtual forces \cite{mavridis2020detection, leonard2001virtual}. 

Alignment and cohesion provide useful team-like properties to the dynamics. It is desirable to construct virtual forces between UAVs which prevent collision between them, hence the repulsion is included. Collision between a UAV and obstacles or confinement to a desired region is enforced by a boundary on their dynamics and encouraged by a confinement potential. The boundary need not be physical for the UAVs: it can be a virtual boundary designed to enforce safety constraints. The self-propulsion is added to counteract the dissipation introduced by the velocity alignment term.

We first study a problem on bounded region $D \subset \mathbb R^3$ with the Yukawa potential providing alignment, attraction, and repulsion. Then, we incorporate leaders into our model, converting these PDEs into a coupled ODE-PDE system. Finally, we discuss Bessel potentials of fractional order $s > 2$.

\subsection{Related Work}

There is existing work on hydrodynamic limits of particle dynamics with fairly 
regular nonlocal alignment and attraction-repulsion kernels over the torus \cite{weak_sols}. 
We change some of the assumptions of this work.  We work with singular alignment, 
attraction, and repulsion kernels and a bounded domain possibly containing obstacles.  Most work on Euler alignment systems with singular interactions focus on the Riesz potential \cite{poyato2017euler, shvydkoy2018eulerian} or a similar mildly singular kernel \cite{an2020global}. \cite{shvydkoy2018eulerian, an2020global} focus on the 1D model in the whole space $\mathbb R$ with singular alignment force only, and obtain global existence of classical solutions.

We modify the work of \cite{weak_sols, feireisl2015weak} using some additions from \cite{donatelli2015well, chiodaroli2016existence}, and, of course, some of our own contributions to develop our approach to show weak solutions exist to our problem. We describe these now.

\cite{weak_sols} additionally deals with an alignment system subject to entropic pressure. In this work, we do not consider pressure. Briefly, we are interested in the Euler alignment system as a limit of particle dynamics. As noted in \cite{weak_sols, carrillo2010particle}, pressure terms appear in hydrodynamic limits of Cucker-Smale particle dynamics where stochastic elastic collisions occur. Slightly different closure assumptions are required, and this results in an Euler alignment system which includes an additional equation for the temperature.

This work is motivated by the need to develop a robust theory for a particular application to system identification problems \cite{mavridis2022learning} in three spatial dimensions \cite{tirumalai2023learning}. Because of this motivation, we are not interested in the problem including pressure and temperature - the Euler alignment system we study already has many interesting features that have not been studied together. It is also for the application in \cite{tirumalai2023learning} that we focus on three spatial dimensions.

\subsection{Approach}

It is known that the compressible Euler equations admit very irregular and ``wild'' solutions \cite{de_lel_weak} which are constructed via convex integration. Convex integration methods are extremely general, and so many problems can be embedded into this framework with some careful modification where necessary \cite{feireisl2015weak, donatelli2015well, chiodaroli2016existence}. In particular, and somewhat counter-intuitively, one is able to embed compressible systems which might include nonlocal interactions \cite{weak_sols} or temperature \cite{donatelli2015well,feireisl2015weak} into abstract \textit{incompressible} systems \cite{feireisl2015weak} in a completely justified and rigorous way. For intuition, see some of the examples in \cite{feireisl2015weak}.

In this work, we show existence of infinitely many weak solutions of our Euler 
alignment system on a bounded, smooth domain. We then show that these solutions satisfy an energy inequality, and are hence bounded-energy weak solutions. We use this energy inequality to 
prove a weak-strong uniqueness principle.

\subsection{Contribution}

To the best of our knowledge, this is one of the first works on this type of 3D Euler alignment system in bounded domains with obstacles. Many other works on similar models are posed on the free space or the torus \cite{ha2014hydrodynamic, tadmor2014critical, shu2020flocking}. In addition, we include a wide variety of interactions in our Euler alignment system, local and non-local, including alignment, attraction, repulsion, obstacles, and self-propulsion. Most work on the Euler alignment system and similar systems focus on a subset of these \cite{poyato2017euler, shvydkoy2018eulerian, an2020global, ha2014hydrodynamic, tadmor2014critical, shu2020flocking}. The (non-fractional) Bessel-type interaction which we select for our alignment, attraction, and repulsion terms has not been studied before. The problem with leaders also has not been studied before. Our expansion of interactions to a wide class of fractional Bessel-type potentials is also novel.

\subsection{Paper Organization}

In Section \ref{notation}, we describe the notation we use throughout the paper. 

Then in Section \ref{section2}, we give a formal (as in relating to the written form of the equations) description of the process one follows to go from a Cuker-Smale-type particle/agent model to an Euler alignment system. This section is only meant to give a conceptual connection between agent models and mean-field models, and should not be understood as a rigorous argument that states the dynamics of the agent models converge to the dynamics of the mean-field model. Additional details can be found in \cite{ha2008particle} for models with fairly regular kernels. Rigorously showing convergence of particle dynamics under singular (i.e. unbounded) interactions to their formal mean-field limits is non-trivial. See \cite{carrillo2014derivation} for discussion of these cases. 

In Section \ref{section3}, we sequentially state and prove each of our results. Each result builds in some way upon the previous. First, we show via convex integration that distributional solutions to the Euler alignment system exist. Then, we show that these solutions satisfy some energy bounds, and so are bounded-energy weak solutions. Finally, we use the energy inequality to show that the solutions satisfy a weak-strong uniqueness principle, i.e. that a Lipschitz solution, should it exist, is unique in the class of bounded-energy weak solutions to the Euler alignment system.

In Section \ref{section4}, we first expand the results we have to include leader-agents which interact with the agent-fluid described by the Euler alignment system, and show that the conclusions of Section 3 hold for this system, as well. We then describe an Euler alignment system subject to a broader class of non-local fractional Bessel-type interactions, and show the conclusions of our previous results in Section 3 hold for this class of models.

Finally, Section \ref{section5} concludes the paper.

\subsection{Notation}
\label{notation}
Throughout this work, we use the ellipsis punctuation `$...$' to indicate line continuation for long equations or expressions. We will also use $C \in \mathbb R^+$ to denote several generic constants, and denote any dependence of this constant on other parameters (constants, functions, or sets) $\alpha_1,...,\alpha_n$ by: $C(\alpha_1,...,\alpha_n)$.
We use the notation $C_w([t_0,t_f];X)$ to mean functions $f:[t_0,t_f] \rightarrow X$, $X$ a reflexive Banach space, s.t. $\forall \{t_n \}_{n=1}^\infty \subset [t_0,t_f]$ : $t_n \rightarrow t \text { in } [t_0,t_f]$ $f$ is s.t.:
\[
( f(t_n) - f(t), \psi )_{X,X^*} \rightarrow 0 \text{ for all } \psi \in X^*,
\]
where $(\cdot,\cdot)_{X,X^*}$ is the pairing of an element of $X$ with an element of its dual space $X^*$.
A similar notation is used in \cite{weak_sols, feireisl2015weak, donatelli2015well}, and a slightly different one in \cite{chiodaroli2016existence, de_lel_weak}. We use the following notations for the inner product and outer products for real vectors $v,w \in \mathbb R^3$:
\[
v^\top w \equiv \langle v,w\rangle_{\mathbb R^3}, 
vw^\top \equiv v \otimes w,
\]
where $v^\top$ is the transposition of $v$. We denote $\text{cl }S$ to be the closure of the set $S \subset \mathbb R^3$, and we use $l:\mathbf L(S)\rightarrow \mathbb R^+_0$ to denote the standard Lebesgue measure, $\mathbf L(S)$ being the $\sigma-$algebra of Lebesgue-measureable sets contained in $S$. If $S$ is bounded and smooth(-enough), $\nu: \partial S \rightarrow \partial B(0,1)$ is the outward pointing normal on the boundary $\partial S$. We use $\mathbb I_{A}:\mathbb R^3\rightarrow \{0,1 \}$ to denote the characteristic (indicator) function of a set $A \subset \mathbf L(\mathbb R^3)$. 

\section{Hydrodynamic Alignment Systems}
\label{section2}
In this section, we introduce a particle model of a multi-agent system and its formal kinetic and hydrodynamic limits. We pass from particles to hydrodynamics only formally.

For certain 
very regular interactions, it can be shown fairly readily that the kinetic and hydrodynamic (under a suitable closure assumption) behaviors indeed correspond to limits related to the empirical measure of the particle dynamics. Singular kernels present significant obstacles to demonstrating such a correspondence \cite{poyato2017euler, carrillo2014derivation}. 

Consider a sequence of particles (boids) 
$\{(x_i, v_i)\}_{i=1}^N (t) \subset \big (\text{cl }D \times \mathbb R^3 \big )$, $t\in [t_0, t_f]$, $D \subset \mathbb R^3$ a bounded, open, connected, smooth region.
The particles are associated to probability space
$\{\Omega,\mathcal B(\Omega), \mathbb P \}$ and state space 
$\{ (\text{cl }D)^N \times \mathbb R^{3N},\mathcal B( (\text{cl }D)^N \times \mathbb R^{3N}) \}$. 
The equations of 
motion for the particles are for each $\omega \in \Omega$:
\begin{equation}
	\begin{cases}
		\frac{d}{dt} x_i(t) = v_i(t) \\
		\frac{d}{dt} v_i(t) = \sum_{\alpha \in A} \Phi_\alpha(x_i(t), v_i(t), 
		\{(x_i, v_i)\}_{i=1}^N(t)) \\
		x_i(t) \in \text{cl } D \text{ }\forall t \in [t_0,t_f]
	\end{cases}
\end{equation}
with i.i.d. initial conditions $(x_i,v_i)(t_0) \sim \mathbb M_{xv} \in \mathcal P(D \times \mathbb R^3)$. 
$A:= \{a,c,r, p,o \}$ are character subscripts corresponding to alignment, cohesion, repulsion, self-propulsion, and obstacle avoidance forces.
These forces $\Phi_\alpha(\cdots)$ are:
\begin{equation}
	\begin{split}
		\label{CSForces}
		&\Phi_a(x_i,v_i, \{(x_i, v_i)\}_{i=1}^N):= \frac{1}{N-1}\sum_{j=1, j \neq i}^N G(x_i, x_j;k_a,\lambda_a) (v_j - v_i), \\
		&\Phi_c(x_i,v_i, \{(x_i, v_i)\}_{i=1}^N):= \frac{1}{N-1}\sum_{j=1, j \neq i}^N \nabla_{1} G(x_i, x_j;k_c,\lambda_c),\\
		&\Phi_r(x_i,v_i, \{(x_i, v_i)\}_{i=1}^N):= \frac{1}{N-1}\sum_{j=1, j \neq i}^N \nabla_{1} G(x_i, x_j;k_r,\lambda_r), \\
		&\Phi_p(x_i,v_i, \{(x_i, v_i)\}_{i=1}^N)= \Phi_p(v_i) :=  \kappa_p v_i(1- P(||v_i||_{\mathbb R^3}))   \\
		&\Phi_o(x_i,v_i, \{(x_i, v_i)\}_{i=1}^N)= \Phi_o(x_i):= \nabla_1 U(x_i) ,
	\end{split}
\end{equation}
where $\nabla_1 (\cdot)$ is the gradient with respect to the first (or only) argument, and 
$k_\alpha, \lambda_{\alpha}$, $\kappa_p > 0$ are real constants. We postpone specification of the class of $G$ for now.
We take: 
\begin{equation}
	\begin{split}
		&P \in C^2_b(\mathbb R;\mathbb R), P'(x) \geq 0 \text{ s.t. } \Phi_p \in \text{Lip}(\mathbb R^3;\mathbb R^3) , \text{ and }U \in C^2_b(D;\mathbb R).
	\end{split}
\end{equation}
To satisfy the spatial constraint, we impose specular reflections on $\partial D$:
\begin{equation}
	\lim_{t^{*+}_{i,j} \downarrow t_{i,j}^*} v_j(t^{*+}_{i,j})=  
	(\mathbf I_3 - 2\mathbf \nu(x_j(t^{*}_{i,j}))\mathbf \nu^\top(x_j(t^{*}_{i,j}))) v_j(t^{*}_{i,j}),  
\end{equation}
where $t^{*}_{i,j}$ is the time of the $i$-th collision of the $j$-th agent with the boundary $\partial D$.

Let $\mu(t) \in \mathcal P(D \times \mathbb R^3), \mu(t_0)=\mathbb M_{xv}$ be the single-particle position-velocity probability distribution \cite{bae2019flocking} of the generalized Cucker-Smale ensemble at time $t$. Assume it has density 
$\rho_{xv}$ w.r.t. the Lebesgue measure. Similarly to \cite{bae2019flocking}, the formal kinetic (Vlasov) limit of the given particle dynamics is:

\begin{equation}
	\begin{cases}
		\partial_t \rho_{xv} + \nabla_x \cdot [v \rho_{xv}] + 
		\nabla_v \cdot [\mathcal F[\rho_{xv}](t,\cdot)\rho_{xv}] = 0 \\
		\mathcal F[\rho_{xv}](t,x,v) := \int_{D \times \mathbb R^3}G(x,y;k_a,\lambda_a)(w-v)
		\rho_{xv}(t,y,w)dy dw + ... \\ 
		\int_{D} \nabla_x G(x,y;k_r,\lambda_r)\rho_{xv}(t,y,v) dy + ... \\
		\int_{D} \nabla_x G(x,y;k_c,\lambda_c)\rho_{xv}(t,y,v) dy + ... \\
		\nabla_x U(x) + \kappa_p v (1 - P(||v||_{\mathbb R^3}))
	\end{cases}
\end{equation}
with boundary conditions on $[t_0,t_f] \times \partial D$ \cite{bae2019flocking, 
	desvillettes2005trend}:
\[
\rho_{xv}(t,x,v) = \rho_{xv}(t,x,(\mathbf I_3 - 2\mathbf \nu(x)\mathbf \nu^\top(x)) v),
\]
and initial condition $\rho_{xv}(t_0,\cdot) = \rho^{xv}_0$, which is the density associated to 
$\mathbb M_{xv}$. 
Now, we define the 
spatial probability density and momentum density:
\[\rho: [t_0, t_f] \times D \rightarrow \mathbb R, \text{ }
\mathbf j: [t_0, t_f] \times D \rightarrow \mathbb R^3,
\] respectively, as:
\begin{equation}
	\begin{split}
		\rho(t,\cdot) &:= \int_{\mathbb R^3} \rho_{xv}(t,\cdot) \text{ } dv \\
		\mathbf j(t,\cdot)&:= \int_{\mathbb R^3} v\rho_{xv}(t,\cdot) \text{ } dv=:\rho(t,\cdot)\mathbf u(t,\cdot).
	\end{split}
\end{equation}
Assume $\rho > 0 \text{ in } D$, and consider $(\rho(t,\cdot),\mathbf j(t,\cdot))^\top$ to be extended by zero outside $D$ when we use them as functions over the whole $\mathbb R^3$.
It can be shown after repeated formal application of integration by parts and 
assumption of the mono-kinetic closure \textit{ansatz} \cite{ballerini2008empirical}:
\begin{equation}
	\rho_{xv}(t,x,v)= \rho(t,x)\delta(v - \mathbf u(t,x))
\end{equation} 
that we 
obtain the compressible Euler equations over $(t,x) \in [t_0,t_f] \times D:$
\begin{equation}
	\begin{cases}\label{euler}
		\partial_t \rho + \nabla_x \cdot \mathbf j = 0 \text{ in } (t_0,t_f] \times D \\
		\partial_t \mathbf j + 
		\nabla_x \cdot [\rho^{-1}\mathbf j \mathbf j^\top] = \mathbf S[\rho, \mathbf j] \text{ in } (t_0,t_f] \times D \\
		\mathbf S[\rho, \mathbf j] =  \rho \mathcal L_a \mathbf j - 
		\mathbf j \mathcal L_a \rho - 
		\rho \nabla_x[\mathcal L_c\rho + \mathcal L_r\rho + U] + 
		\kappa_p \mathbf j(1-P(\rho^{-1}||\mathbf j||_{\mathbb R^3} ) ),
	\end{cases}
\end{equation}
subject to the non-penetrative boundary conditions \cite{desvillettes2005trend} on  $\partial D$ and initial conditions:
\begin{equation}
	\label{bcs}
	\mathbf j^\top \mathbf \nu = 0 \text{ } \forall 
	\text{ } (t,x) \in [t_0,t_f] \times \partial D; \text{ } (\rho(t_0,\cdot), \mathbf j(t_0,\cdot))^\top = (\rho_0, \mathbf j_0)^\top \text{ in } D.
\end{equation}
This is the Euler alignment system.
$\mathcal L_{\alpha}: L^2(D;\mathbb R) \rightarrow L^2(D;\mathbb R)$
denotes a linear operator of form
\[
\mathcal L_\alpha \phi \circ (t,\cdot):= \int_{\mathbb R^3} 
G(\cdot,y;k_\alpha;\lambda_\alpha)\mathbb I_D(y) \phi(t,y) \text{ } dy.
\]
If we take 
\[
G(\cdot;k_a,\lambda_a), G(\cdot;k_c,\lambda_c), 
G(\cdot;k_r,\lambda_r)
\] 
to be fundamental solutions for 
elliptic partial differential operators $\Lambda_a, \Lambda_c,
\Lambda_r$, respectively, 
we can construct an augmented system of PDEs for 
state vector field $(\rho,\mathbf j)$ with
alignment potentials $\pi_\rho[\rho], 
\pi_{\mathbf j}[\mathbf j]$
and cohesion-repulsion potentials $\mathbf V[\rho]:=(v_c[\rho], v_r[\rho])^\top$:
\begin{equation}
	\begin{cases}\label{euler2}
		\partial_t \rho + \nabla_x \cdot \mathbf j = 0 \text{ in } (t_0,t_f] \times D; \\
		\partial_t \mathbf j + \nabla_x \cdot [\rho^{-1}\mathbf j \mathbf j^\top] = 
		\mathbf S[\rho, \mathbf j] \text{ in } (t_0,t_f] \times D; \\
		\Lambda_a \mathbf \pi_\rho[\rho](t,\cdot) = \mathbb I_D \rho(t,\cdot) \text{ in } \mathbb R^3;\\
		\Lambda_a \mathbf \pi_{\mathbf j}[\mathbf j](t,\cdot)= \mathbb I_D \mathbf j(t,\cdot) \text{ in } \mathbb R^3; \\
		\Lambda_c v_c[\rho](t,\cdot) = \mathbb I_D \rho(t,\cdot) \text{ in } \mathbb R^3;  \\
		\Lambda_r v_r[\rho](t,\cdot) = \mathbb I_D \rho(t,\cdot) \text{ in } \mathbb R^3; \\
		\mathbf S[\rho, \mathbf j] := \rho \pi_{\mathbf j}[\mathbf j] - \mathbf j \pi_\rho[\rho] - \rho \nabla_x[\mathbf 1^\top_2 \mathbf V[\rho] + U] 
		+ 
		\kappa_p \mathbf j (1-P(\rho^{-1}||\mathbf j||_{\mathbb R^3} )); \\
		\mathbf j^\top \mathbf \nu = 0 \text{ in } [t_0,t_f] \times \partial D; \\ 
		(\rho(t_0,\cdot), \mathbf j(t_0,\cdot))^\top = (\rho_0, \mathbf j_0)^\top \text{ in } D,
	\end{cases}
\end{equation}
where $\mathbf 1_n$ is the $n$-vector with ones as its entries. 
We pick the parametrized family of linear elliptic operators with constant coefficients:
\begin{equation}\label{bvps}
	\mathcal L_\alpha^{-1} = \Lambda_\alpha =\frac{1}{4\pi k_\alpha}
	(\lambda_\alpha^2 I - \nabla_x^2),
\end{equation}
with $k_a, k_r > 0, k_c < 0$, $\lambda_a, \lambda_r, \lambda_c \neq 0$.
We present a very simple lemma which will aid us in showing existence of weak 
solutions to this augmented system later on. 
\newline
\begin{lemma}
	\label{reg_pots}
	Consider the elliptic problem:
	\begin{equation*}
		\begin{split}
			\Lambda_\alpha w[\phi] & = \mathbb I_D \phi \text{ in } \mathbb R^3 
		\end{split}
	\end{equation*}
	with $k_\alpha, \lambda_\alpha \neq 0$. 
	If $\phi \in L^{\infty}(D;\mathbb R)$, then 
	$w \in C^1_b (D;\mathbb R)$.
	If: \[\phi \in C^{0,\alpha}(D;\mathbb R),\alpha \in (0,1]\] then $\phi 
	\in C^{2,\alpha}(D;\mathbb R)$.
\end{lemma}
\begin{proof}
	(Sketch.) Notice the fundamental solution is the Yukawa potential:
	\[
	G(||x-y||_{\mathbb R^3};k_\alpha, \lambda_\alpha) = k_\alpha \frac{e^{-\lambda_a 
			||x-y||_{\mathbb R^3}}}{||x-y||_{\mathbb R^3}}
	= 4 \pi k_\alpha e^{-\lambda_a 
		||x-y||_{\mathbb R^3}} G_N(||x-y||_{\mathbb R^3})
	\]
	where $G_N$ is the Newtonian potential. The Newtonian potential 
	clearly is also an upper bound to the Yukawa potential. 
	A solution to the above is:
	\begin{equation}
		\label{conv}
		w[\phi](\cdot) =(G(\cdot;k_\alpha,\lambda_\alpha)*\mathbb I_D v)(\cdot) =\int_{\mathbb R^3}G(\cdot - y;k_\alpha,\lambda_\alpha)\mathbb I_D(y)\phi(y) \text{ } dy.
	\end{equation}
	We can use the same approaches to obtain regularity for the Newtonian potential \cite{gilbarg2015elliptic} on the Yukawa potential 
	and the results of Lemma \ref{reg_pots} follow.
\end{proof}
\begin{remark}
	We can prove this Lemma via results on Bessel potential spaces and conclude $w[v] \in C^{2,\alpha}(D;\mathbb R)$ assuming $v \in H^2(D;\mathbb R)$ and then using a Sobolev embedding to obtain H\"older regularity. See Theorem \ref{bessel_pot} for details.
\end{remark}
Hereafter, assume the potentials $\pi_\rho, \pi_{\mathbf j}, v_r, v_c$ are computed via the given convolution integral representation (\ref{conv}) for each time $t \in [t_0,t_f]$.
\section{Existence and Weak-Strong Uniqueness}
\label{section3}
In this section, we describe weak solutions to (\ref{euler2}). The solutions $(\rho,\mathbf j)$ are of type:
\begin{equation}
	\begin{split}
		\label{class_of_sols}
		&\rho \in C^3([t_0,t_f];H^2(D;\mathbb R)) 
		\cap \text{Lip}([t_0,t_f] \times \text{cl }D;\mathbb R), \rho > 0,   \\
		&\mathbf j \in 
		C_{w}([t_0, t_f]; L^2(D;\mathbb R^3)) \cap 
		L^\infty((t_0,t_f) \times D;\mathbb R^3)
	\end{split}
\end{equation}
which satisfy (\ref{euler2}) in the following sense:

\begin{equation}
	\label{weak1}
	\int_{D} \rho(t,\cdot) \phi(t,\cdot) \text{ } dx \text{ } \Big |_{t_0}^{t_f} = 
	\int_{t_0}^{t_f} \int_{D} \rho \partial_t \phi + \mathbf j^\top
	\nabla_x \phi \text{ }dx \text{ } dt;
\end{equation}
\begin{equation}
	\label{weak2}
	\begin{split}
		&\int_{D} \mathbf j^\top(t,\cdot) \psi(t,\cdot) \text{ } dx \text{ } \Big |_{t_0}^{t_f} = ... \\
		&\int_{t_0}^{t_f} \int_{D} \mathbf j^\top \partial_t \psi + 
		\langle \rho^{-1} \mathbf j \mathbf j^\top,
		\nabla_x \psi \rangle_{\mathbb R^{3 \times 3}} 
		+ \psi^\top {\mathbf S}
		[\rho, \mathbf j] \text{ }dx \text{ } dt;
	\end{split}
\end{equation}
\begin{equation}
	\label{weak3}
	\rho(t_0,\cdot) = 
	\rho_0 , 
	\mathbf j(t_0,\cdot)  = \mathbf j_{0}
\end{equation}
\newline
for arbitrary $(\phi, \psi) \in \mathcal D([t_0,t_f] \times D;\mathbb R \times \mathbb R^3)$. Throughout, assume:
\[
\rho_0 \in H^2(D;\mathbb R) \cap \text{Lip}(\text{cl }D;\mathbb R), \rho_0 > 0,
\]
\[
\mathbf j_0 \in \text{Lip}(\text{cl }D;\mathbb R^3),
\]
and consider $(\rho(t,\cdot),\mathbf j(t,\cdot))^\top$ to be extended by zero outside $D$ if we must use them as functions over the whole $\mathbb R^3$.
\subsection{Existence of ``Wild'' Weak Solutions}
\begin{theorem}
	\label{prop1}  
	Consider the problem (\ref{euler2}). 
	Suppose $\rho_0$, $\mathbf j_0$, $U$, $P$ are (as assumed earlier):
	\[
	\begin{split}
		&P \in C^2_b(\mathbb R;\mathbb R), P'(x) \geq 0 \text{ s.t. } \Phi_p \in \text{Lip}(\mathbb R^3;\mathbb R^3) , \text{ and }U \in C^2_b(D;\mathbb R),
	\end{split}
	\]
	\[
	\rho_0 \in H^2(D;\mathbb R) \cap \text{Lip}(\text{cl }D;\mathbb R), \rho_0 > 0,
	\]
	\[
	\mathbf j_0 \in \text{Lip}(\text{cl }D;\mathbb R^3).
	\]
	Then, (\ref{euler2}) admits infinitely many weak solutions of the type:
	\begin{equation}
		\begin{split}
			&\rho \in C^3([t_0,t_f];H^2(D;\mathbb R)) 
			\cap \text{Lip}([t_0,t_f] \times \text{cl }D;\mathbb R), \rho > 0,   \\
			&\mathbf j \in 
			C_{w}([t_0, t_f]; L^2(D;\mathbb R^3)) \cap 
			L^\infty((t_0,t_f) \times D;\mathbb R^3). 
		\end{split}
	\end{equation}
\end{theorem}
\begin{proof}
	To prove this result, we embed the problem into an abstract functional \textit{incompressible} Euler 
	system \cite{feireisl2015weak, donatelli2015well, chiodaroli2016existence} and apply results therefrom. First, define the binary operation $(\cdot \odot \cdot) : \mathbb R^3 \times 
	\mathbb R^3 \rightarrow \mathbb S^3_{0}$, the zero-trace symmetric matrices:
	\[
	v \odot w:= vw^\top - \frac{1}{3}v^\top w \mathbf I_{3}.
	\]
	We seek to transform our problem into: Find 
	\[
	\mathbf v \in \Big ( C_w([t_0,t_f];L^2(D; \mathbb R^3)) 
	\cap L^\infty((t_0,t_f) \times D;\mathbb R^3) \Big ) =: X
	\]
	s.t. (for some $\mathbf h, r, \mathbf M$ of a class soon-to-be-specified):
	\begin{equation}
		\begin{split}
			\label{abs1}
			\partial_t \mathbf v + &\nabla_x \cdot \Big [ \frac{(\mathbf v + \mathbf h[\mathbf v]) 
				\odot (\mathbf v + \mathbf h[\mathbf v]) }{r[\mathbf v]} + \mathbf M[\mathbf v]\Big ] = 0
		\end{split}
	\end{equation}
	in the sense of $\mathcal D'([t_0,t_f] \times D;\mathbb R^3)$;
	\begin{equation}
		\nabla_x \cdot \mathbf v(t,\cdot) = 0 
	\end{equation}
	(after an extension by zero to $\mathbb R^3$) in the sense of $\mathcal D'([t_0,t_f] \times \mathbb R^3; \mathbb R^3)$;
	\begin{equation}
		\begin{split}
			\label{abs2}
			\mathbf v^\top(t,\cdot) \mathbf \nu  =  0 \text{ on } \partial D 
			\text{ } \forall \text{ }
			t \in [t_0,t_f];
		\end{split}
	\end{equation}
	\begin{equation}
		\begin{split}
			\label{abs3}
			\mathbf v(t_0,\cdot) = \mathbf v_0;
		\end{split}
	\end{equation}
	subject to the constraint:
	\begin{equation}
		\begin{split}
			\label{abs4}
			\frac{1}{2}\frac{||\mathbf v + \mathbf h[\mathbf v]||^2_{\mathbb R^3}}{r[\mathbf v]}(\cdot) = 
			e[\mathbf v](\cdot) \text{ a.e. in } [t_0,t_f] \times D
		\end{split}
	\end{equation}
	for some $e$ whose class is likewise soon-to-be-specified.
	It is assumed 
	for open set $Q \subseteq (t_0,t_f) \times D \text{ s.t. } l(Q) = l((t_0,t_f) \times D)$ that:
	\[
	\mathbf h: X \rightarrow C_b(Q;\mathbb R^3),
	\]
	\[
	r: X 
	\rightarrow C_b(Q;\mathbb R), r > 0,
	\]
	$$e: X  \rightarrow C_b(Q;\mathbb R), e \geq 0,\text{ and }$$
	$$\mathbf M: X \rightarrow C_b(Q;\mathbb S^3_0)$$
	are $Q$\textit{-continuous}, which is to say:
	\begin{definition}
		An operator 
		\[
		b: X \rightarrow C_b(Q;\mathcal V),
		\]
		$\mathcal V$ a finite dimensional vector space, is $\mathbf{Q}$\textbf{-continuous}, if $b$ is s.t.:
		\begin{enumerate}
			\item $b$ maps sets bounded in $L^\infty((t_0,t_f) \times D;\mathbb R^3)$ to sets bounded in $C_b(Q;\mathcal V)$;
			\item For $\{\mathbf w_m \}_
			{m=1}^\infty$, $\mathbf w_m \in X$, and 
			$\mathbf w_m \rightarrow \mathbf w \in X$ weakly 
			in $C_w([t_0,t_f]; L^2(D;\mathbb R^3))$ and 
			weakly-* in $L^\infty((t_0,t_f) \times D; 
			\mathbb R^3)$ implies
			\[
			b[\mathbf w_m] \rightarrow b[\mathbf w] \text{ uniformly in } C_b(Q; \mathcal V);
			\]
			\item If $\mathbf v(t,\cdot) = \mathbf w(t,\cdot)$ for $t_0 \leq t \leq 
			\tau \leq t_f$, then:
			\[
			b[\mathbf v] = b[\mathbf w] \text{ in } \big ((0,\tau] \times D\big ) \cap Q.
			\]
		\end{enumerate}
	\end{definition}
	We introduce a set of subsolutions to the abstract system (\ref{abs1}- \ref{abs3}), $X_0 \subset X,$
	where for $\mathbf F \in L^\infty((t_0,t_f) \times D; 
	\mathbb S^3_0), \mathbf v \in X$, 
	$\mathbf h, \mathbf F, r, e, \mathbf M$ are $Q$-continuous and s.t for 
	$t_0 < t < t_f$: 
	\begin{equation}
		\label{abs5}
		\frac{3}{2}\lambda_{max}\Big [ \frac{(\mathbf v + \mathbf h[\mathbf v]) 
			\odot (\mathbf v + \mathbf h[\mathbf v]) }{r[\mathbf v]} + \mathbf M[\mathbf v] - \mathbf F \Big] (t,\cdot)
		-e[\mathbf v](t,\cdot) < 0,
	\end{equation}
	\[
	X_0:= \Big \{ \mathbf v : (\ref{abs2}- \ref{abs5}) \text{ satisfied}, 
	\partial_t \mathbf v + \nabla_x \cdot \mathbf F = 0 \text{ in } 
	\mathcal D'([t_0,t_f] \times D;\mathbb R^3) \Big \}.
	\]
	A pair $(\mathbf v, \mathbf F)$ which satisfy these constraints is termed an admissible subsolution-flux pair.
	
	Similarly to \cite{weak_sols, feireisl2015weak}, we have the following theorem:
	\begin{theorem} 
		\label{existence_fundamental}    
		Suppose operators $\mathbf h, r, e, \text{ and }\mathbf M$ are $Q$-continuous and subject to the conditions (\ref{abs5}). 
		Suppose that $(r[\mathbf v])^{-1}$ is 
		$Q$-continuous.
		Further, suppose $X_0$ is non-empty and bounded in $L^\infty((t_0,t_f) \times D; 
		\mathbb R^3)$. Then, the abstract problems (\ref{abs1}- \ref{abs4}) admit 
		infinitely many weak solutions.
	\end{theorem}
	\begin{proof}
		Define 
		\[
		e[\mathbf v](t,\cdot):= \sigma(t) - \frac{3}{2} \Sigma[\mathbf v](t,\cdot)
		\]
		for some $Q$- continuous operator $\Sigma$, and some $\sigma:[t_0,t_f] 
		\rightarrow \mathbb R$ to be determined s.t. $e[\mathbf v] \geq 0$, and 
		\[
		\eta[\mathbf v, \mathbf F](t,\cdot):= \frac{3}{2} \lambda_{max} \Big [ 
		\frac{(\mathbf v + \mathbf h[\mathbf v])(\mathbf v + \mathbf h[\mathbf v])^\top}{
			r[\mathbf v]} + \mathbf M[\mathbf v] - \mathbf F
		\Big ](t,\cdot).
		\]
		To begin the proof, we present a lemma, given very similarly to Lemma 3.2 in \cite{chiodaroli2016existence}:
		\begin{lemma}
			Suppose $\{(\mathbf v_m, \mathbf F_m) \}_{m=1}^\infty$ are a sequence 
			of admissible subsolution-flux pairs. Then, $\exists$ a unique 
			$(\mathbf v^*, \mathbf F^*)$ s.t. 
			\begin{equation}
				\begin{split}
					&\mathbf v_m \rightarrow \mathbf v^* \text{ in } C_w([t_0,t_f];L^2(D;\mathbb R^3)), \text{ and weakly-* in }
					L^\infty((t_0,t_f)\times D;\mathbb R^3); \\
					& \mathbf F_m \rightarrow \mathbf F^* \text{ weakly-* in } 
					L^\infty((t_0,t_f)\times D;\mathbb S_0^3),
				\end{split}
			\end{equation}
			which is an admissible subsolution-flux pair to the abstract Euler problem up to the relaxation on 
			$t_0 < t < t_f$:
			\begin{equation}
				\frac{3}{2}\lambda_{max}\Big [ \frac{(\mathbf v^* + \mathbf h[\mathbf v^*]) 
					\odot (\mathbf v^* + \mathbf h[\mathbf v^*]) }{r[\mathbf v^*]} + \mathbf M[\mathbf v^*] - \mathbf F^* \Big](t,\cdot) 
				-e[\mathbf v^*](t,\cdot) \leq 0.
			\end{equation}
			Further, if for the functional $I:L^2([t_0,t_f]\times D
			;\mathbb R^3) \times (0, (t_f-t_0)/2) \rightarrow \mathbb R^-_0$:
			\begin{equation}
				I[\mathbf v](t):= 
				\inf_{\tau \in [t_0 + t, t_f - t]} \int_{D} 
				\frac{1}{2} \frac{||\mathbf v + \mathbf h[\mathbf v]||^2_{\mathbb R^3}}{r[\mathbf v]}(\tau,\cdot) - e[\mathbf v](\tau,\cdot) \text{ }dx
			\end{equation}
			$I[\mathbf v^*](t) \equiv 0$, then $\mathbf v^*$ solves the abstract Euler system 
			(\ref{abs1}- \ref{abs4}).
		\end{lemma}
		\begin{proof} (sketch)
			The claim follows with minor modification to the argument for Lemma 3.2. in \cite{chiodaroli2016existence}. Simply replace the construction of the 
			operator $e[\cdot]$ and the functional $I[\cdot]$ therein with the ones defined here and the claim follows.
		\end{proof}
		Next, we give the oscillatory lemma of \cite{feireisl2015weak, chiodaroli2016existence}.
		\begin{lemma}(Oscillatory Lemma)
			\label{oscil_lemma}
			Let $\mathbf v \in X_0$ and $(\mathbf v, \mathbf F)$ be an admissible subsolution-flux pair, and suppose $t \in (t_0,t_f)$. Then, $\exists \{\mathbf w_m \}_{m=1}^\infty \subset \mathcal D((t,t_f) \times D;\mathbb R^3), 
			\{\mathbf \Psi_m\}_{m=1}^\infty \subset \mathcal D((t,t_f) \times D;\mathbb S^3_0)$ s.t.
			\begin{equation*}
				\begin{split}
					&(\mathbf v + \mathbf w_m, \mathbf F + \mathbf \Psi_m) \text{ are an admissible subsolution-flux pair } \\
					&\mathbf w_m \rightarrow 0 \text{ in } C_w([t_0,t_f];L^2(D;\mathbb R^3)) \\
					& \text{there exists } C(E) > 0 \text{ s.t. }  e[\mathbf v](t) \leq E \text{ and } \\
					& \liminf_{m \rightarrow \infty} I[\mathbf v + 
					\mathbf w_m](t) \geq I[\mathbf v](t) + C(E)I^2[\mathbf v](t).
				\end{split}
			\end{equation*}
		\end{lemma}
		\begin{proof}
			Let $\tilde X:= C_w([t_0,t_f];L^2(D;\mathbb R^3)) \cap C((t_0,t_f) 
			\times D;\mathbb R^3)$, and:
			\[
			f \in C_b((t_0,t_f) \times D;\mathbb R); \text{ } \mathbf \Gamma \in C_b((t_0,t_f) \times D;\mathbb S_0^3),
			\]
			\begin{equation}
				\begin{split}
					\Xi(\mathbf \Gamma, f):= \Big\{ &\mathbf v \in  \tilde X : \text{ there exists } \mathbf \Upsilon \in 
					C_b((t_0,t_f) \times D;\mathbb S_0^3) \text{ s.t. } ...  \\ 
					& \partial_t \mathbf v + \nabla_x \cdot \mathbf \Upsilon = 0 \text{ in } \mathcal D'([t_0,t_f] \times D;\mathbb R^3), ... \\
					& \mathbf v^\top \mathbf \nu = 0 \text{ on } \partial D \text{ }
					\forall \text{ }t \in [t_0,t_f], ... \\
					& \forall \tau > t_0 
					\inf_{(t,x)\in (\tau, t_f) \times D} f(t,x) - 
					\epsilon[\mathbf v, \mathbf \Upsilon, \mathbf \Gamma](t,x) > 0 , ... \\ 
					& \mathbf v(t_0,\cdot) = \mathbf v_0 \Big \},
				\end{split}
			\end{equation}
			where $\epsilon:\tilde X \times C_b((t_0,t_f) \times D;\mathbb S^3_0) \times C_b((t_0,t_f) \times D;\mathbb S^3_0) 
			\rightarrow C_b(Q;\mathbb R)$ is
			\begin{equation}
				\epsilon[\mathbf v, \mathbf \Upsilon, \mathbf \Gamma]:= 
				\frac{3}{2}\lambda_{max}(\frac{1}{r[\mathbf v]}\mathbf v \mathbf v^\top + \mathbf \Gamma - \mathbf \Upsilon).
			\end{equation}
			
			To continue the proof of the oscillatory lemma, we present a lemma similar to how it is presented in \cite{chiodaroli2016existence}, 
			which was originally proven in \cite{donatelli2015well}:
			\begin{lemma}
				Let $\mathcal A:=(t_1,t_2) \times D \subseteq Q$ be an open subset, and $\mathbf \Gamma, f$ of the classes previously specified, $f > 0$ in $(t_0,t_f) \times D$. Suppose 
				$\mathbf v \in \Xi(\mathbf \Gamma, f)$ with flux $\mathbf \Upsilon$. Then, there exists
				$ \Lambda > 0$ depending only on $f$ and sequences 
				$\{\mathbf w_m \}_{m=1}^\infty \subset \mathcal D(\mathcal A;\mathbb R^3), \{\mathbf \Psi_m \}_{m=1}^\infty \subset \mathcal D(\mathcal A;\mathbb S^3_0)$ 
				s.t.
				\begin{equation}
					\begin{split}
						&\mathbf v + \mathbf w_m \in \Xi(\mathbf \Gamma, f) \text{ with flux } \mathbf \Upsilon + \mathbf \Psi_m\\
						&\mathbf w_m \rightarrow 0 \text{ in } C_w([t_0,t_f];L^2(D;\mathbb R^3)) \\
						&\liminf_{m \rightarrow \infty} \inf_{t \in (t_1,t_2)} 
						\int_{D}\frac{1}{2}\frac{1}{r[\mathbf v]}||\mathbf v + \mathbf w_m||_{\mathbb R^3}^2 \text{ } dx\geq ... \\
						&\inf_{t \in (t_1,t_2)} \Big [
						\int_{D}\frac{1}{2}\frac{1}{r[\mathbf v]}||\mathbf v||_{\mathbb R^3}^2 \text{ }dx + 
						\Lambda \Big (\int_{D} f - \frac{1}{2} \frac{1}{r[\mathbf v]}||\mathbf v||_{\mathbb R^3}^2 \text{ } dx \Big ) ^2
						\Big ]
					\end{split}
				\end{equation}
			\end{lemma}
			The rest of the proof of Lemma \ref{oscil_lemma} follows directly from the arguments given in \cite{chiodaroli2016existence}.
		\end{proof}
		Finally, we use the results on Baire categories first used by 
		\cite{de_lel_weak} and applied in \cite{feireisl2015weak, chiodaroli2016existence}, and others:
		\begin{lemma}
			Let $(X,d)$ be a complete metric space and suppose function 
			$I:X \rightarrow \mathbb R^-_0$ is a Baire class-1 function. Suppose 
			$X_0 \subseteq X$ is dense in $X$ and non-empty, s.t.
			$\forall \beta < 0 \exists \alpha(\beta) > 0$, $\forall x \in X_0$ with:
			\[
			I(x) < \beta < 0 \text{ } \exists \{ x_m\}_{m=1}^\infty \in X_0, 
			x_m \rightarrow x \text{ in } (X,d)
			\]
			and
			\[
			\liminf_{m\rightarrow \infty} I(x_m) \geq I(x) + \alpha(\beta)
			\] 
			then $\exists S \subseteq X_0$ s.t. $I(x) = 0 \text{ on } S$.
		\end{lemma}
		We now follow the rest of the result of \cite{chiodaroli2016existence} and the proof of Theorem \ref{existence_fundamental} is concluded.
	\end{proof}
	Now, it remains to construct the operators 
	$r, \mathbf h, e, \mathbf M$ as they pertain to our problem. Take 
	$Q:= (t_0,t_f) \times D$. We take the 
	Helmholtz decomposition for $\mathbf j$:
	\[
	\mathbf j = \mathbf v + \nabla_x \Phi[\mathbf v]
	\]
	similarly to the procedure of \cite{li2021weak}.
	To satisfy the BCs on $\mathbf j$ as well as the continuity equation for $\rho$, 
	we take $\Phi[\mathbf v]$ to solve the Neumann problem:
	\begin{equation}
		\label{neumann}
		\begin{split}
			\nabla_x^2 \Phi[\mathbf v] = -\partial_t \rho[\mathbf v] \text { in } D \\
			(\nabla_x^\top \Phi[\mathbf v]) \mathbf \nu = 0 \text{ on } \partial D.
		\end{split}
	\end{equation}
	Using the Helmholtz decomposition, we have expressed the 
	momentum in such a way that we can select $\rho[\mathbf v]$ as we need it. 
	Select arbitrary
	\[
	\rho \in C^3([t_0,t_f];H^2(D;\mathbb R)) 
	\cap \text{Lip}([t_0,t_f] \times \text{cl }D;\mathbb R), \rho > 0, \rho(t_0,\cdot) = \rho_0
	\] 
	s.t. $\int_{D}\rho(t,\cdot) \text{ }dx = 1$. Now, rather obviously using the dominated convergence theorem on one hand:
	\[
	\frac{d}{dt}\int_{D} \rho(t,\cdot) \text{ } dx = \int_{D} \partial_t \rho(t,\cdot) \text{ } dx,
	\]
	and on the other hand we know $\int_{D}\rho(t,\cdot)dx = 1$. So, we conclude
	\[
	\int_{D} \partial_t \rho(t,\cdot) \text{ }dx = 0.
	\]
	Since the R.H.S. of (\ref{neumann}) 
	and the associated boundary 
	condition have the same integral over their respective domains for each $t \in [t_0,t_f]$, the Neumann problem (\ref{neumann}) 
	has a 
	unique $H^1(D;\mathbb R)$ solution up to an additive constant for 
	each $t \in [t_0,t_f]$ via Lax-Milgram \cite{gilbarg2015elliptic}. It is sufficient to select any one of these solutions. Additionally, we have 		      	that: 
	\[
	\Phi[\mathbf v](t,\cdot) \in H^4(D;
	\mathbb R) \subset C^{2,1/2}(D;
	\mathbb R)
	\] using elliptic regularity and Sobolev 
	embedding \cite{taylor2010partial, evans}. $\nabla_x \Phi[\mathbf v]$ is thus unique. From thrice time-differentiability of 
	$\rho[\mathbf v]$: 
	\[
	\Phi[\mathbf v] \in C^2([t_0,t_f];  C^{2,1/2}(D;
	\mathbb R)).
	\]
	The Helmholtz decomposition of $\mathbf j_0$ is then:
	\[
	\mathbf j_0 = \mathbf v_0 + \nabla_x \Phi[\mathbf v](t_0,\cdot)
	\]
	so we can obtain $\mathbf v_0$ from $\rho_0, \mathbf j_0$.
	The assumption on $\rho[\mathbf v]$ guarantees that via the integral representation outlined in Lemma \ref{reg_pots}, 
	the repulsion and attraction potentials are s.t. (via a Sobolev embedding): 
	\begin{equation}
		v_r[\rho[\mathbf v]](t,\cdot), v_c[\rho[\mathbf v]](t,\cdot)
		\in C^{2,1/2}(D;\mathbb R) \text{ for each } t\in [t_0,t_f],
	\end{equation}
	so the gradients of these potentials exist in the classical sense.
	Note that also:
	\[
	\pi_\rho[\rho[\mathbf v]](t,\cdot)
	\in C^{2,1/2}(D;\mathbb R), \pi_{\mathbf j}[\mathbf v + \nabla_x \Phi[\mathbf v]] \in C^{2,1/2}(D;\mathbb R^3)  \text{ for each } t\in [t_0,t_f]
	\]
	via Lemma \ref{reg_pots} by the assumptions on $\rho[\mathbf v](t,\cdot)$ and $\mathbf v \in X$.
	We re-write the momentum equation using the Helmholtz decomposition:
	\begin{equation}
		\label{euler3}
		\begin{split}
			\partial_t \mathbf v + & \nabla_x \cdot \Big [ \frac{(\mathbf v + 
				\nabla_x \Phi[\mathbf v]) 
				(\mathbf v + \nabla_x \Phi[\mathbf v])^\top }{\rho[\mathbf v]} + 
			\mathbf I \partial_t \Phi [\mathbf v] \Big ]  = 
			\mathbf S[\cdots]  \\
			&\mathbf S[\rho[\mathbf v], \mathbf v + 
			\nabla_x \Phi[\mathbf v] ]= 
			\rho[\mathbf v] \pi_{\mathbf j} \Big [\mathbf v + 
			\nabla_x \Phi[\mathbf v] \Big ] - ... \\
			&(\mathbf v + \nabla_x \Phi[\mathbf v]) \pi_{\rho} [\rho[\mathbf v]] -
			\rho[\mathbf v] \nabla_x[\mathbf 1_2^\top \mathbf V[\rho[\mathbf v]] + U] + ... \\ 
			& \kappa_p \mathbf 
			(\mathbf v + \nabla_x \Phi[\mathbf v])(1-P(\rho^{-1}[\mathbf v]||
			\mathbf v + \nabla_x \Phi[\mathbf v]||_{\mathbb R^3} ) ).
		\end{split}
	\end{equation}
	Let:
	\[
	\mathbf h[\mathbf v] = 
	\nabla_x \Phi[\mathbf v],
	\]
	which is clearly $Q$-continuous (it is constant in $\mathbf v$).
	From Theorem \ref{existence_fundamental}, 
	we want $\mathbf v$ over $Q$ to be subject to the constraint:
	\[
	\frac{1}{2}\frac{|| \mathbf v  + 
		\mathbf h[\mathbf v]||_{\mathbb R^3}^2}{\rho[\mathbf v]} 
	= 
	e[\mathbf v](t,\cdot) \equiv \sigma[\mathbf v](t) - \frac{3}{2} \Sigma[\mathbf v](t,\cdot) \text{ a.e. on } Q.
	\]
	Take $\Sigma[\mathbf v] = \partial_t \Phi[\mathbf v]$, and assume $\sigma(\cdot)$ is continuous and bounded. We will specify $\sigma$ further later on.
	From the R.H.S., $e[\mathbf v]$ is clearly $Q$-continuous.
	We modify the momentum equation again using this ``kinetic energy'' :
	\begin{equation}
		\label{euler4}
		\begin{split}
			&\partial_t \mathbf v + 
			\nabla_x \cdot \Big [ \frac{(\mathbf v + 
				\mathbf h[\mathbf v]) 
				\odot (\mathbf v + \mathbf h[\mathbf v]) }{\rho[\mathbf v]} \Big ]  = 
			\mathbf S[\cdots]  \\
			&\mathbf S[\rho[\mathbf v], \mathbf v + 
			\mathbf h[\mathbf v] ] = 
			\rho[\mathbf v] \pi_{\mathbf j} \Big [\mathbf v + 
			\mathbf h[\mathbf v] \Big ] - ... \\
			& (\mathbf v + \mathbf h[\mathbf v]) \pi_{\rho}[\rho[\mathbf v]] - 
			\rho[\mathbf v] \nabla_x[\mathbf 1_2^\top \mathbf V[\rho[\mathbf v]] + U] + ... \\
			&\kappa_p \mathbf 
			(\mathbf v + \mathbf h[\mathbf v])(1-P(\sqrt{2}\rho^{-1/2}[\mathbf v]
			e^{1/2}[\mathbf v] ) ).
		\end{split}
	\end{equation}
	Consider now the second-order elliptic homogeneous Dirichlet problem :
	\begin{equation}
		\begin{split}
			\label{elliptic}
			-&\nabla_x \cdot \Big [ \nabla_x \mathbf w[\mathbf v] + \nabla_x^\top \mathbf w[\mathbf v] 
			- \frac{2}{3} \mathbf I \text{ tr }\nabla_x \mathbf w[\mathbf v]
			\Big] = \mathbf S[\rho[\mathbf v], \mathbf j[\mathbf v]]
			\text{ in } D; \\
			&\mathbf w[\mathbf v] = 0 \text{ on } \partial D
		\end{split}
	\end{equation}
	which is linear in $\mathbf w$. Let
	\begin{equation}
		\label{M_matrix}
		\mathbf M[\mathbf v]= \nabla_x \mathbf w[\mathbf v] + \nabla_x^\top \mathbf w[\mathbf v] 
		-  \frac{2}{3} \mathbf I \text{tr } \nabla_x \mathbf w[\mathbf v],
	\end{equation}
	which is trace-free symmetric.
	Define bilinear form $\mathbf B: H^1_0(D;\mathbb R^{3}) \times 
	H^1_0(D;\mathbb R^3) \rightarrow \mathbb R$:
	\[
	\mathbf B[ \chi, \mathbf \phi]:= \big \langle 
	\nabla_x \chi + \nabla_x^\top \chi - 
	\frac{2}{3} \mathbf I \text{ tr }\nabla_x \mathbf \chi, \nabla_x \phi
	\big \rangle_{L^2(D;\mathbb R^{3 \times 3})}.
	\]
	Over matrices of form $\xi \eta^\top, \xi, \eta \in \mathbb R^3$, we can show using 
	the elementary symmetric polynomials of the matrix characteristic polynomial:
	\[
	\langle \xi \eta^\top + \eta \xi^\top - \frac{2}{3}\text{tr } \xi \eta^\top \mathbf I, \xi \eta^\top 
	\rangle_{\mathbb R^{3 \times 3}}
	= 2||\xi \eta^\top||_{\mathbb R^{3 \times 3}}^2 - \frac{2}{3}[\text{tr }  \xi \eta^\top ]^2 
	= \frac{4}{3}||\xi \eta^\top||_{\mathbb R^{3 \times 3}}^2 
	\]
	since these matrices are rank-1. Hence, the (constant-coefficient) elliptic operator in the given problem 
	satisfies the 
	Legendre-Hadamard condition, so by G$\overset{\circ}{\text{a}}$rding's theorem \cite{pdesys}, the associated
	bilinear form is coercive. Now, we apply the Lax-Milgram theorem to conclude that the problem:
	\[
	\mathbf B[\chi,\phi] = \big( \mathbf S[\rho[\mathbf v], \mathbf v + \mathbf h[\mathbf v]], \phi \big )_{H^{-1}(D;\mathbb R^{3}), H^1_0(D;\mathbb R^{3})}
	\]
	with $\mathbf S[\rho[\mathbf v](t,\cdot), \mathbf j[\mathbf v](t,\cdot)] \in L^{2}( D;\mathbb R^3)$ 
	has a unique $H^1_0(D;\mathbb R^3)$ solution $\chi$ for each $t \in [t_0,t_f]$. Now, we apply 
	this result to $\mathbf w[\mathbf v]$, and conclude that $\mathbf w[\mathbf v] 
	\in H^1_0( D;\mathbb R^3)$ exists and is unique. 
	By boundary regularity \cite{pdesys} of the constant-coefficient problem, $\mathbf w[\mathbf v](t,\cdot) \in H^2_0(D;\mathbb R^3)$. $\mathbf w$ is also a continuous linear map from $L^p(D;\mathbb R^3)$ into $W^{2,p}(D;\mathbb R^3)$. See Lemma 7.1., Corr. 7.2. of \cite{chiodaroli2016existence} for details. Since $S[\rho[\mathbf v], \mathbf v + \mathbf h[\mathbf v]]$ is also $L^\infty((t_0,t_f) \times 
	D;\mathbb R^3)$, we conclude  
	$\mathbf w[\mathbf v] \in L^\infty(t_0,t_f;W^{2,p}(D;\mathbb R^3))$ for $p \in (1,\infty)$.
	It can be verified that the operator $\mathbf S$ reformulated as we have done using $\Phi, e$ preserves weak continuity in $L^2(D;\mathbb R^3)$ and weak$^*$ continuity in $L^\infty(D;\mathbb R^3)$ for each $t \in [t_0,t_f]$. Since $\mathbf v \in C_w([t_0,t_f];L^2(D;\mathbb R^3))$, $\mathbf w \in C_w([t_0,t_f];L^2(D;\mathbb R^3))$ by continuity, so $\mathbf w$ is s.t. $\forall \psi \in L^2(D;\mathbb R^3)$:
	\[
	\infty > \langle \mathbf w(t,\cdot)\Big |_{t_0}^{t_f},\psi\rangle_{L^2(D;\mathbb R^3)} = 
	\int_D \int_{t_0}^{t_f} \partial_t (\mathbf w^\top(t,\cdot)\psi) \text{ }dt \text{ } dx = ...
	\]
	\[
	\int_{t_0}^{t_f} \int_D \partial_t \mathbf w^\top(t,\cdot)\psi \text{ }dx \text{ }dt,
	\] 
	so $\partial_t \mathbf w \in L^2(t_0,t_f;L^2(\tilde D;\mathbb R^3))$. We have the embeddings \cite{adams2003sobolev}:
	\[
	W^{2,4}(D;\mathbb R^3) \subset \subset C^1_b(D;\mathbb R^3) \subset L^2(D;\mathbb R^3).
	\]
	By the Aubin-Lions-Simon theorem (Cor. 4 of  \cite{simon1986compact}),
	\[
	\mathbf w[X] \subset \subset C([t_0,t_f];C^1_b(D;\mathbb R^3)),
	\]
	so $\mathbf M \in C_b(Q;\mathbb S^3_0)$ $\forall \mathbf v \in X$. The last property for $Q$-continuity of this operator is easily verified.
	Hence, 
	we finally have:
	\begin{equation}
		\label{euler6}
		\begin{split}
			&\partial_t \mathbf v + 
			\nabla_x \cdot \Big [ \frac{(\mathbf v + \mathbf h[\mathbf v]) 
				\odot (\mathbf v + \mathbf h[\mathbf v]) }{\rho[\mathbf v]} + \mathbf M[\mathbf v]\Big ]  = 0.
		\end{split}
	\end{equation}
	$r[\mathbf v]:=\rho[\mathbf v]$ is obviously $Q$-continuous, as is $(r[\mathbf v])^{-1}$, 
	and, as we have noted, so are $e$, $\mathbf h$, and $\mathbf M$.
	Let $\mathbf F \equiv 0$, i.e. $\mathbf v(t,\cdot) = \mathbf v_0$ a.e..
	So, the set $X_0$ is non-empty by our choices. We can select a 
	function $\sigma \in C_b([t_0,t_f];\mathbb R)$ s.t.:
	\begin{equation}
		\begin{split}
			\sup_{(t,x) \in Q, t > \tau} \frac{3}{2}\Big [\lambda_{max}\Big [ \frac{(\mathbf v + \mathbf h[\mathbf v]) 
				\odot (\mathbf v + \mathbf h[\mathbf v]) }{r[\mathbf v]} + \mathbf M[\mathbf v]\Big] 
			+  \partial_t \Phi[\mathbf v] \Big ]< 
			\sigma[\mathbf v](t)
		\end{split}
	\end{equation}
	since the L.H.S. is bounded. 
	Now, since we have selected $\sigma$, we employ the inequality similar to the one given in 
	Lemma 3.2 of \cite{de_lel_weak}:
	\[
	\frac{1}{2}\frac{||\mathbf v + \mathbf h[\mathbf v]||^{2}_{\mathbb R^3}}{r} 
	\leq \frac{3}{2} \lambda_{max} \Big (\frac{(\mathbf v + \mathbf h[\mathbf v]) (\mathbf v + \mathbf h[\mathbf v])^\top}{r} - 
	\mathbf U  \Big ) \text{ if } \mathbf U \in \mathbb S^3_0
	\]
	to conclude $X_0$ is bounded in $L^\infty((t_0,t_f) \times D;\mathbb R^3)$. We invoke Theorem \ref{existence_fundamental}  and the claim of Theorem \ref{prop1} is proven.
\end{proof}
\subsection{Bounded-Energy Weak Solutions}
We now consider a more specific variety of weak solution to this problem: the 
bounded-energy weak solution. We construct a quantity which we show is bounded along 
solutions to the the system (\ref{euler2}). First, we show some formal computations to motivate construction of the energy inequality similarly to \cite{weak_sols, feireisl2015weak}. Recall $\mathbf u = \rho^{-1}\mathbf j$. We begin:
\[
\partial_t [\rho \mathbf u^\top \mathbf u] = [\partial_t (\rho \mathbf u)]^\top \mathbf u + 
(\rho \mathbf u)^\top \partial_t \mathbf u.
\]
Multiplying the momentum equation in (\ref{euler2}) by $\mathbf u$ gives:
\[
\begin{split}
	& \mathbf u^\top \partial_t (\rho \mathbf u) + \mathbf u^\top \nabla_x \cdot [\rho \mathbf u 
	\mathbf u^\top] = \rho \mathbf u^\top \pi_{\mathbf j}[\rho \mathbf u] - 
	\rho \mathbf u^\top \mathbf u \pi_{\rho}[\rho] - ... \\
	& \rho \mathbf u^\top \nabla_x [\mathbf 1_2^\top \mathbf V[\rho] + U] 
	+ \kappa_p \rho \mathbf u^\top \mathbf u
	[1 - P(||\mathbf u||_{\mathbb R^3})].
\end{split}
\]
Substituting using the relation for $\partial_t [\rho \mathbf u^\top \mathbf u]$ yields:
\[
\begin{split}
	&\partial_t [\rho \mathbf u^\top \mathbf u] - 
	(\rho \mathbf u)^\top \partial_t \mathbf u + 
	\mathbf u^\top \nabla_x \cdot [\rho \mathbf u 
	\mathbf u^\top] = ... \\
	&\rho \mathbf u^\top \pi_{\mathbf j}[\rho \mathbf u] - \rho \mathbf u^\top \mathbf u \pi_{\rho}[\rho] -
	\rho \mathbf u^\top \nabla_x [\mathbf 1_2^\top \mathbf V[\rho] + U] 
	+ \kappa_p \rho \mathbf u^\top \mathbf u
	[1 - P(||\mathbf u||_{\mathbb R^3})].
\end{split}
\]
From the continuity and momentum equations, with $\rho > 0$, formally:
\[
\partial_t \mathbf u + (\nabla_x^\top \mathbf u)\mathbf u = 
\pi_{\mathbf j}[\rho \mathbf u] - \mathbf u \pi_{\rho}[\rho] - 
\nabla_x [\mathbf 1_2^\top \mathbf V[\rho] + U] + \kappa_p \mathbf u(1 - P(||\mathbf u||_{\mathbb R^3})).
\]
Substituting these terms into the previous computation and applying integration by parts formally, we obtain:
\[
\begin{split}
	\frac{1}{2} \frac{d}{dt}\int_{D}\rho \mathbf u^\top \mathbf u \text{ }dx 
	&= \int_{D}\rho \mathbf u^\top \pi_{\mathbf j}[\rho \mathbf u] - 
	\rho \mathbf u^\top \mathbf u \pi_{\rho}[\rho]  \text{ }dx + ... \\
	&
	\int_{D }-\rho \mathbf u^\top \nabla_x [\mathbf 1_2^\top \mathbf V[\rho] + U] 
	+ \kappa_p \rho \mathbf u^\top \mathbf u
	[1 - P(||\mathbf u||_{\mathbb R^3})] \text{ }dx.
\end{split}
\]
Notice that formally:
\[
\int_{D} \rho \mathbf u^\top \nabla_x v_r[\rho] \text{ }dx = - 
\int_{D} \nabla_x \cdot [\rho \mathbf u]  v_r[\rho] \text{ }dx 
= \frac{1}{2}\frac{d}{dt}\int_{D}  \rho v_r[\rho] \text{ }dx
\]
and similarly for $v_c$.
So, we now obtain the (strict) energy balance after re-expressing the 
equations in $(\rho, \mathbf j)$:
\begin{equation}
	\label{energy}
	\begin{split}
		&\frac{1}{2} \frac{d}{dt}\int_{D} \rho^{-1} 
		\mathbf j^\top \mathbf j
		+ \rho v_r[\rho] +  \rho v_c[\rho] \text{ }dx 
		= ... \\ 
		&\int_{D}\mathbf j^\top \pi_{\mathbf j}[\mathbf j] - 
		\rho^{-1} 
		\mathbf j^\top \mathbf j \pi_{\rho}[\rho] + 
		\mathbf j^\top \nabla_x U + 
		\kappa_p \rho^{-1} \mathbf j^\top \mathbf j 
		[1 - P(||\rho^{-1}\mathbf j||_{\mathbb R^3})] \text{ }dx.
	\end{split}
\end{equation}
We have now constructed a suitable energy quantity to study. We now relax the strictness of what we obtained to accommodate weak solutions to the system (\ref{euler2}). Similarly to \cite{weak_sols, feireisl2015weak, diss_weak}, we say:
\newline
\newline
\begin{definition}
	\label{def_3_8}
	A pair $(\rho, \mathbf j)$ is a bounded-energy
	weak solution to (\ref{euler2}) iff:
	\begin{enumerate}
		\item $(\rho,\mathbf j)$ are of type: 
		\[
		\begin{split}
			&\rho \in C^3([t_0,t_f];H^2(D;\mathbb R))\cap \text{Lip}([t_0,t_f] \times \text{cl }D;\mathbb R) , \rho > 0;  \\
			&\mathbf j \in 
			X;
		\end{split}
		\]
		\item The solutions satisfy (\ref{weak1}-\ref{weak3});
		
		\item The solution additionally satisfies the relaxed energy balance:
		\[
		\begin{split}
			&\frac{1}{2}\int_{D} \rho^{-1} 
			\mathbf j^\top \mathbf j(t,\cdot)
			+ \rho v_r[\rho](t,\cdot) +  \rho v_c[\rho](t,\cdot) \text{ }dx \text{ }\Big |_{t_0}^\tau
			\leq ... \\ 
			&\int_{t_0}^\tau \int_{D}\mathbf j^\top \pi_{\mathbf j}[\mathbf j] -
			\rho^{-1} 
			\mathbf j^\top \mathbf j \pi_{\rho}[\rho] \text{ }dx \text{ }dt + ... \\ 
			&\int_{t_0}^\tau \int_{D} \mathbf j^\top \nabla_x U +
			\kappa_p \rho^{-1} \mathbf j^\top \mathbf j 
			\Big (1 - P(||\rho^{-1}\mathbf j||_{\mathbb R^3}) \Big ) \text{ }dx \text{ }dt
		\end{split}
		\]
		a.e. on $[t_0,t_f]$.
	\end{enumerate}
\end{definition}
\begin{theorem}
	\label{prop2}  
	Consider the problem (\ref{euler2}) . 
	Suppose $\rho_0$, $\mathbf j_0$, $U$, $P$ are (as assumed earlier):
	\[
	\begin{split}
		&P \in C^2_b(\mathbb R;\mathbb R), P'(x) \geq 0 \text{ s.t. } \Phi_p \in \text{Lip}(\mathbb R^3;\mathbb R^3) , \text{ and }U \in C^2_b(D;\mathbb R),
	\end{split}
	\]
	\[
	\rho_0 \in H^2(D;\mathbb R) \cap \text{Lip}(\text{cl }D;\mathbb R), \rho_0 > 0,
	\]
	\[
	\mathbf j_0 \in \text{Lip}(\text{cl }D;\mathbb R^3).
	\]
	Then, (\ref{euler2})
	admits infinitely many bounded-energy weak solutions.
\end{theorem}
\begin{proof}
	As in Theorem \ref{prop1}, we can select 
	densities $\rho$ to be as regular as we desire. So, as before, select 
	$\rho \in C^3([t_0,t_f];H^2(D;\mathbb R))\cap \text{Lip}([t_0,t_f] \times \text{cl }D;\mathbb R)$. Then, construct potential 
	$\Phi[\mathbf v]$, and we 
	conclude that there are infinitely many fields $\mathbf v$ 
	which satisfy the abstract Euler problem in the weak sense. 
	From the construction of the energy operator $e[\mathbf v]$, we know that a.e. on $[t_0,t_f] \times D$:
	\begin{equation}
		\frac{1}{2}\frac{||\mathbf j||_{\mathbb R^3}^2}{\rho}(t,x) = e[\mathbf v](t,x) \equiv \sigma[\mathbf v](t) - 
		\frac{3}{2}\partial_t \Phi[\mathbf v](t,x). 
	\end{equation}
	So, we know that:
	\[
	\begin{split}
		\frac{1}{2} \frac{d}{dt}\int_{D} \rho^{-1} 
		\mathbf j^\top \mathbf j + 
		\rho v_r[\rho] + \rho v_c[\rho] \text{ }dx  = ... \\
		\frac{1}{2} \frac{d}{dt}\int_{D} 2(\sigma(t) - 
		\frac{3}{2} \partial_t \Phi[\mathbf v] ) + 
		\rho v_r[\rho] + \rho v_c[\rho] \text{ } dx.
	\end{split}
	\]
	Considering the R.H.S. of the strict energy balance (\ref{energy}), the 
	last term can be bounded easily as:
	\[
	\begin{split}
		\kappa_p \int_{D} \rho^{-1} 
		\mathbf j^\top \mathbf j &\Big ( 1 - P(||\rho^{-1}\mathbf j||_{\mathbb R^3}) \Big ) \text{ } dx \leq C(\kappa_p, P) \int_{D} \rho^{-1} 
		\mathbf j^\top \mathbf j  \text{ } dx = ... \\ 
		&2C(\kappa_p, P) \int_{D} 
		(\sigma[\mathbf v](t) - 
		\frac{3}{2} \partial_t \Phi[\mathbf v] ) \text{ } dx,
	\end{split}
	\]
	as $(1-P)$ is uniformly bounded.
	For the second-to-last term, we note that there exist 
	$\tilde \rho, \bar \rho$ s.t. 
	\[
	0 < \tilde \rho < \rho 
	< \bar \rho < \infty \text{ everywhere in } [t_0,t_f] \times D,
	\]
	and so:
	\[
	(\bar \rho)^{-1/2} ||f ||_{L^2(D;\mathbb R^3)}   
	\leq  ||\rho^{-1/2}(t,\cdot) f ||_{L^2(D;\mathbb R^3)}  
	\leq (\tilde \rho)^{-1/2} ||f||_{L^2(D;\mathbb R^3)}  
	\]
	for $f \in L^2(D;\mathbb R^3)$. $||\rho^{-1/2}(t,*)(\cdot)||_{L^2(D;\mathbb R^3)}$ additionally satisfies the properties of norms 	 - it is obviously non-negative, $||\rho^{1/2}(t,\cdot) f||_{L^2(D;\mathbb R^3)} = 0$ iff $f \equiv_{L^2(D;\mathbb R^3)} 0$, and satisfies: 
	\[||\rho^{-1/2}(t,\cdot) (f+g)||_{L^2(D;\mathbb R^3)} \leq ||\rho^{-1/2}(t,\cdot) f||_{L^2(D;\mathbb R^3)} + ||\rho^{-1/2}(t,\cdot) g||_{L^2(D;\mathbb R^3)}\]
	\[||C\rho^{-1/2}(t,\cdot) f||_{L^2(D;\mathbb R^3)} \leq |C| ||\rho^{-1/2}(t,\cdot) f||_{L^2(D;\mathbb R^3)}.
	\]
	So, $||\rho^{-1/2}(t,*)(\cdot)||_{L^2(D;\mathbb R^3)}$ is equivalent to the standard norm on $L^2(D;\mathbb R^3)$. Accordingly, using the Cauchy-Schwarz inequality, the equivalence we obtained, and Young's inequality:
	\[
	\int_{D} \mathbf j^\top \nabla_x U \text{ } dx 
	\leq C(U) ||\mathbf j||_{L^2(D;\mathbb R^3)}
	\leq C(U)\bar{\rho}^{1/2}||\rho^{-1/2}(t,\cdot) \mathbf j(t,\cdot) ||_{L^2(D;\mathbb R^3)} \leq ...
	\]
	\[
	\frac{1}{2}C(U,\bar \rho)^2 + \frac{1}{2}\int_{D} 
	\rho^{-1}\mathbf j^\top \mathbf j \text{ } dx = \frac{1}{2}C(U,\tilde \rho)^2 + \int_{D} 
	(\sigma[\mathbf v](t) - 
	\frac{3}{2} \partial_t \Phi[\mathbf v] ) \text{ } dx.
	\]
	On the alignment terms, we can apply the triangle inequality, Young's convolution inequality, our equivalence of norms, and
	the boundedness of $\pi_{\rho}[\rho]$ to obtain:
	\[
	\int_{D}\mathbf j^\top \pi_{\mathbf j}[\mathbf j] -
	\rho^{-1} 
	\mathbf j^\top \mathbf j \pi_{\rho}[\rho] \text{ } dx 
	\leq C(k_a,\lambda_a,\tilde \rho, \bar \rho) ||\rho^{-1/2}(t,\cdot) \mathbf j(t,\cdot) ||_{L^2(D;\mathbb R^3)}^2  =...
	\]
	\[
	2C(k_a,\lambda_a,\tilde \rho, \bar \rho) \int_{D} 
	(\sigma[\mathbf v](t) - 
	\frac{3}{2} \partial_t \Phi[\mathbf v] ) \text{ } dx
	\]
	so we arrive at:
	\[
	\begin{split}
		\frac{1}{2} \frac{d}{dt}\int_{D} 2(\sigma[\mathbf v](t) - 
		&\frac{3}{2} \partial_t \Phi[\mathbf v] ) + 
		\rho v_r[\rho] + \rho v_c[\rho] \text{ } dx \leq ... \\
		&C_1(\kappa_p, P, U,  \tilde \rho, \bar \rho, k_a, \lambda_a) (1 + \int_{D} 
		(\sigma[\mathbf v](t) - 
		\frac{3}{2} \partial_t \Phi[\mathbf v] ) \text{ } dx).
	\end{split}
	\]
	We move the potential terms to the R.H.S., which gives:
	\[
	\begin{split}
		&\frac{d}{dt}\int_{D} \sigma[\mathbf v](t) \text{ } dx = l(D) \frac{d}{dt}\sigma[\mathbf v](t) 
		\leq ... \\ 
		&C_1(\cdots) \Big (1 + \int_{D} 
		\sigma[\mathbf v](t) + 
		|\frac{3}{2} \partial_t \Phi[\mathbf v]|  \text{ } dx \Big ) + ... \\
		&\Big | \frac{1}{2}\frac{d}{dt}\int_{D} \rho v_r[\rho] + 
		\rho v_c[\rho] - 3\partial_t \Phi[\mathbf v] \text{ } dx \Big | 
		\leq 
		C_2(D,C_1(\cdots))(1 + \sigma[\mathbf v](t))
	\end{split}
	\]      
	by time-differentiability and spatial boundedness in $D$ of the relevant functions. These spatial bounds depend only on the choice of $\tilde \rho, \bar \rho$ which the relevant bounds on $\rho$. Thus, selection of $\sigma$ which does not increase `too quickly' to satisfy a bound on the energy is feasible. It remains to choose such a $\sigma$. We do so similarly to \cite{weak_sols}
	Let $\sigma_0, b$ be positive, real constants sufficiently large so that the differential inequality we obtained is satisfied by $\sigma$ of form:
	\[
	\sigma[\mathbf v](t) = \sigma_0 + \text{exp}(-b(t-t_0)) > \frac{3}{2}||\partial_t \Phi[\mathbf v]||_{C_b([t_0,t_f] \times D;\mathbb R)},
	\]
	so $e[\mathbf v](t) > 0$, and that
	\begin{equation}
		\begin{split}
			&\frac{1}{2} \frac{d}{dt}\int_{D} \rho^{-1} 
			\mathbf j^\top \mathbf j
			+ \rho v_r[\rho] +  \rho v_c[\rho] \text{ } dx 
			\leq ... \\ 
			&\int_{D}\mathbf j^\top \pi_{\mathbf j}[\mathbf j] - 
			\rho^{-1} 
			\mathbf j^\top \mathbf j \pi_{\rho}[\rho] + 
			\mathbf j^\top \nabla_x U +
			\kappa_p \rho^{-1} \mathbf j^\top \mathbf j 
			[1 - P(||\rho^{-1}\mathbf j||_{\mathbb R^3})] \text{ } dx.
		\end{split}
	\end{equation}
	It remains to examine if the L.H.S. is discontinuous at $t_0$. To conclude it is not, we use a slightly modified version of Theorem 6.1 of \cite{feireisl2015weak} given as Lemma 6.2 in \cite{chiodaroli2016existence}. The proof is identical to the one given in \cite{feireisl2015weak}.
	\begin{theorem}
		Let $\mathbf M$ as defined in (\ref{M_matrix}) satisfy the hypotheses of Theorem \ref{existence_fundamental}, and suppose 
		\[
		l(\{x \in D : (t,x) \in Q \}) = l(D).
		\] 
		Then, $\exists$ a dense subset 
		$\mathcal R \subset (t_0,t_f)$ s.t. $\forall t \in \mathcal R$ 
		$\exists$ $\mathbf v \in X$ s.t.:
		\[
		\mathbf v \in C(((t_0,t)\cup (t,t_f))\times D;\mathbb R^3) 
		\cap C_w([t_0,t_f];L^2(D;\mathbb R^3));
		\]
		\[
		\mathbf v(t_0,\cdot) = \mathbf v_0;
		\]
		\[
		\exists \mathbf F \in C(((t_0,t)\cup (t,t_f))\times D;\mathbb S^3_{0})  \text{ s.t. } (\mathbf v, \mathbf F)
		\]
		is an admissible subsolution-flux pair;
		\[
		(e[\mathbf v] - \eta[\mathbf v, \mathbf F])(t,x) > 0 \text{ in }
		((t_0,t)\cup (t,t_f))\times D;
		\]
		\[
		\frac{1}{2}\frac{|| \mathbf v  + 
			\mathbf h[\mathbf v]||_{\mathbb R^3}^2}{\rho[\mathbf v]} 
		= 
		e[\mathbf v] \text{ a.e. in }
		((t_0,t)\cup (t,t_f))\times D.
		\]
	\end{theorem}
	Thus, the claim of Theorem \ref{prop2} follows with $Q = (t_0,t_f) \times D$.
\end{proof}
\subsection{Weak-Strong Uniqueness}
We have established that there are solutions which satisfy energy bounds. Now, we show that if there exists a strong solution within the larger class of bounded-energy weak solutions, it is unique.

Our approach is slightly different than the usual ones taken for weak-strong uniqueness \cite{}. Generally, for compressible Euler systems similar to \ref{euler2} with (infinitely many) distributional solutions $(\rho, \mathbf j)$, we would show that any strong solution $(r,\mathbf i)$ would be unique in the class of weak solutions, assuming it exists.

However, we are not even interested in the case where $\rho$ is distributional - our definition of weak solution in (\ref{class_of_sols} - \ref{weak3}) is already restricted to only admit strong solutions for $\rho$. So, our weak-strong uniqueness principle is only for $\mathbf j$, reflected in our choice of relative energy.

\begin{theorem}
	\label{prop3}  
	Consider the problem (\ref{euler2}), and suppose $\rho_0$, $\mathbf j_0$, $U$, $P$ are (as assumed earlier):
	\[
	\begin{split}
		&P \in C^2_b(\mathbb R;\mathbb R), P'(x) \geq 0 \text{ s.t. } \Phi_p \in \text{Lip}(\mathbb R^3;\mathbb R^3) , \text{ and }U \in C^2_b(D;\mathbb R),
	\end{split}
	\]
	\[
	\rho_0 \in H^2(D;\mathbb R) \cap \text{Lip}(\text{cl }D;\mathbb R), \rho_0 > 0,
	\]
	\[
	\mathbf j_0 \in \text{Lip}(\text{cl }D;\mathbb R^3).
	\] Select a density of type: 
	\[
	\rho \in C^3([t_0,t_f];H^2(D;\mathbb R)) 
	\cap \text{Lip}([t_0,t_f] \times \text{cl }D;\mathbb R), \rho > 0, \rho(t_0,\cdot) = \rho_0,
	\]
	and apply the results of Theorem \ref{prop2} to obtain infinitely many bounded-energy weak solutions $(\rho,\mathbf j)$. Then, assuming there exists a momentum field
	\[
	(r, \mathbf i) \in \text{Lip}([t_0,t_f] \times \text{cl }D;\mathbb R \times \mathbb R^3), 
	(r,\mathbf i)(t_0,\cdot) = (\rho_0, \mathbf j_0)
	\] 
	which is a strong solution to (\ref{euler2}), every $(\rho, \mathbf j)$ corresponding to a bounded-energy weak solution is s.t.
	\[
	(\rho,\mathbf j) = (r, \mathbf i) \text{ a.e. in } [t_0,t_f] \times D.
	\]
\end{theorem}
\begin{proof}
	To prove this result, we define a relative energy quantity. 
	Let $\mathbf u = \rho^{-1}\mathbf j$, 
	$\tilde {\mathbf u}:= r^{-1} \mathbf i$. Note that $\tilde{ \mathbf u}\in \text{Lip}([t_0,t_f] \times \text{cl }D;\mathbb R^3)$.
	Let \[\mathbf E:\Big ( \text{Lip}([t_0,t_f] \times \text{cl }D;\mathbb R) \times X \Big ) \times \Big ( \text{Lip}([t_0,t_f] \times \text{cl }D;\mathbb R \times \mathbb R^3) \Big ) \times [t_0,t_f] 
	\rightarrow \mathbb R\]
	be:
	\[
	\mathbf E[(\rho, \mathbf j), (r, \mathbf i)](t):= 
	\frac{1}{2}\int_{D} \rho(t,\cdot) ||\mathbf u(t,\cdot) 
	- \tilde{\mathbf  u}(t,\cdot) ||_{\mathbb R^3}^2 \text{ } dx + \frac{1}{2}||\rho(t,\cdot) - r(t,\cdot)||_{L^2(D;\mathbb R)}^2.
	\]
	$||\rho^{1/2}(t,\cdot) f||_{L^2(D;\mathbb R^3)}$ with $f \in L^2(D;\mathbb R^3)$ is in fact a norm - it is obviously non-negative, $||\rho^{1/2}(t,\cdot) f||_{L^2(D;\mathbb R^3)} = 0$ iff $f \equiv_{L^2(D;\mathbb R^3)} 0$, and satisfies: 
	\[||\rho^{1/2}(t,\cdot) (f+g)||_{L^2(D;\mathbb R^3)} \leq ||\rho^{1/2}(t,\cdot) f||_{L^2(D;\mathbb R^3)} + ||\rho^{1/2}(t,\cdot) g||_{L^2(D;\mathbb R^3)}\]
	\[||C\rho^{1/2}(t,\cdot) f||_{L^2(D;\mathbb R^3)} \leq |C| ||\rho^{1/2}(t,\cdot) f||_{L^2(D;\mathbb R^3)},
	\]	
	so it is a norm over $L^2(D;\mathbb R^3)$, and indeed, using the constants $\tilde \rho, \bar \rho$ which we defined earlier in the proof of Theorem \ref{prop2}:	
	\[
	(\tilde \rho)^{1/2} ||f ||_{L^2(D;\mathbb R^3)}   
	\leq  ||\rho^{1/2}(t,\cdot) f ||_{L^2(D;\mathbb R^3)}  
	\leq (\bar\rho)^{1/2} ||f||_{L^2(D;\mathbb R^3)},
	\]
	hence it is equivalent to the standard norm over $L^2(D;\mathbb R^3)$.
	Now, as in \cite{weak_sols}, we need to determine suitable 
	remainder terms to construct estimates of $\mathbf E[\cdots](t)|_{t_0}^\tau, \tau \in [t_0,t_f]$. To that end, first assume that $(r, \mathbf i)$ are also a bounded-energy weak solution to the Euler alignment system, i.e. they satisfy Definition \ref{def_3_8} in addition to being a strong solution s.t.:
	\[
	\partial_t r + \nabla_x \cdot \mathbf i = 0 \text{ a.e. in } [t_0,t_f] \times D.
	\]
	
	We copy the energy inequality now expressed in $(\rho,\mathbf u)$ (suppressing arguments):
	\[
	\begin{split}
		\frac{1}{2}\int_{D}\rho {\mathbf u}^\top {\mathbf u} + 
		\rho v_r[\rho] &+ 
		\rho v_c[\rho] \text{ } dx \Big |_{t_0}^\tau \leq ... \\
		&\int_{t_0}^\tau \int_{D} \rho {\mathbf u}^\top \pi_{\mathbf j}[\rho 
		{\mathbf u} ] - \rho {\mathbf u}^\top {\mathbf u} \pi_{\rho}[\rho] + 
		\rho \mathbf u^\top \nabla_x U \text{ } dx \text{ } dt + ... \\
		& \int_{t_0}^\tau \int_{D}  \kappa_p \rho \mathbf u^\top \mathbf u  
		\Big (1 - P(||\mathbf u||_{\mathbb R^3}) \Big ) \text { } dx \text { } dt.
	\end{split} 
	\]
	
	We have:
	\[
	\frac{1}{2}\int_{D}\rho v_r[\rho] \text{ } dx \text{ }  \Big |_{t_0}^\tau = \int_{t_0}^{\tau}\int_{D} \partial_t \rho v_r[\rho] \text{ } dx \text{ } dt = 
	-\int_{t_0}^{\tau}\int_{D} \rho \mathbf u^\top \nabla_x v_r[\rho] \text{ } dx \text{ } dt, 
	\]
	and similarly for the term containing $v_c$,
	so the energy inequality becomes:
	\[
	\begin{split}
		\frac{1}{2}\int_{D}\rho {\mathbf u}^\top {\mathbf u} \text{ } dx \Big |_{t_0}^\tau &\leq \int_{t_0}^\tau \int_{D} \rho {\mathbf u}^\top \pi_{\mathbf j}[\rho 
		{\mathbf u} ] - \rho {\mathbf u}^\top {\mathbf u} \pi_{\rho}[\rho] + 
		\rho \mathbf u^\top \nabla_x U \text{ } dx \text{ } dt + ... \\
		& \int_{t_0}^\tau \int_{D}  
		\kappa_p \rho \mathbf u^\top \mathbf u  
		\Big (1 - P(||\mathbf u||_{\mathbb R^3})\Big ) \text{ } dx \text{ } dt + ...\\ 
		& \int_{t_0}^{\tau}\int_{D} \rho \mathbf u^\top \nabla_x v_r[\rho] \text{ } dx \text{ } dt + \int_{t_0}^{\tau}\int_{D} \rho \mathbf u^\top \nabla_x v_c[\rho] \text{ } dx \text{ } dt.
	\end{split} 
	\]
	These computations were rather formal, but can be justified rigorously by testing on one hand the standard-mollified potentials $v_r^\epsilon := v_r[\rho]*\eta(\cdot;\epsilon)$ against the continuity equation and passing to the limit $\epsilon \rightarrow 0$, and on the other hand applying the fundamental theorem of calculus and radial symmetry of the convolution kernel. 
	
	To begin estimating $\mathbf E[\cdots](t)|_{t_0}^\tau, \tau \in [t_0,t_f]$, similarly to \cite{feireisl2015weak}, first multiply the momentum 
	equation by $\tilde{\mathbf u}$ and integrate:
	\[
	\begin{split}
		\int_{D} \rho {\mathbf u}^\top \tilde{\mathbf u} \text{ } dx \text { }
		\Big |_{t_0}^\tau = &\int_{t_0}^\tau \int_{D}
		\rho \mathbf u^\top \partial_t \tilde{\mathbf u} + \langle \rho {\mathbf u} {\mathbf u}^\top, 
		\nabla_x \tilde{\mathbf u} \rangle_{\mathbb R^{3 \times 3}} \text{ } dx \text { } dt - ... \\
		&\int_{t_0}^\tau \int_{D} \rho  \nabla_x^\top( v_r[\rho] + v_c[\rho] + U) \tilde{\mathbf u} \text{ } dx \text { } dt + ... \\
		&\int_{t_0}^\tau \int_{D} 
		\rho \tilde{\mathbf u}^\top \pi_j[\rho {\mathbf u}] - 
		\rho {\mathbf u}^\top \tilde{\mathbf u} \pi_{\rho}[\rho] \text{ } dx \text { } dt + ... \\
		& \int_{t_0}^\tau \int_{D} \kappa_p \rho \mathbf u^\top \tilde{\mathbf u} \Big (1-P(||\mathbf u||_{\mathbb R^3})\Big ) \text{ }dx \text { } dt.
	\end{split}
	\]
	Then, multiply the continuity equation by 
	$\frac{1}{2}\tilde{\mathbf u}^\top \tilde{\mathbf u}$ and integrate:
	\[
	\begin{split}
		\frac{1}{2}\int_{D} \rho \tilde{\mathbf u}^\top \tilde{\mathbf u} \text{ } dx \text { }
		\Big |_{t_0}^\tau = &\int_{t_0}^\tau \int_{D}
		\rho  \tilde{\mathbf u}^\top \partial_t \tilde {\mathbf u} + 
		\rho \mathbf u^\top (\nabla_x^\top \tilde{\mathbf u}) \tilde{\mathbf u} 
		\text{ } dx \text { } dt.
	\end{split}
	\]
	\newline
	Again, these computations were a bit formal, but we can justify them as before using a mollification procedure on $\tilde{\mathbf u}, \text{ and }\frac{1}{2}\tilde{\mathbf u}^\top \tilde{\mathbf u}$.
	The evolution of the $\rho \mathbf u^\top \mathbf u$ term is provided by the energy inequality. 
	Then, combining these computations with the energy inequality, we obtain:
	\[
	\mathbf E[(\rho,\mathbf j),(r, \mathbf i)](t) \Big |_{t_0}^\tau 
	\leq 
	\int_{t_0}^\tau \mathbf R[(\rho,\mathbf j),(r, \mathbf i)](t) \text { } dt
	\]
	where $\mathbf R:\Big (\text{Lip}([t_0,t_f] \times \text{cl }D;\mathbb R) \times X \Big) \times \text{Lip}([t_0,t_f] \times \text{cl }D;\mathbb R \times \mathbb R^3) \times [t_0,t_f] 
	\rightarrow \mathbb R$ is similarly to \cite{weak_sols, feireisl2015weak}:
	\[
	\begin{split}
		&\mathbf R[(\rho,\mathbf j),(r, \mathbf i)](\cdot) := \int_{D}
		\rho [\partial_t \tilde{\mathbf u} +  (\nabla_x^\top \tilde{\mathbf u})\mathbf u]^\top
		(\tilde{\mathbf u} - {\mathbf u}) \text { } dx + ... \\ 
		&\int_{D} \rho  (\mathbf u - \tilde{\mathbf u})^\top
		\nabla_x v_r [\rho]  + \rho  (\mathbf u - \tilde{\mathbf u})^\top
		\nabla_x v_c [\rho] \text { } dx + ... \\
		&\int_{D} ( \pi^\top_{\mathbf j}[
		\rho {\mathbf u}] -  {\mathbf u}^\top 
		\pi_{\rho}[\rho])\rho({\mathbf u} -  \tilde{\mathbf u} ) \text { } dx + ... \\
		&\int_{D} \rho (\mathbf u - \tilde{\mathbf u})^\top \nabla_x U \text { } dx + ... \\ 
		& \int_{D }\kappa_p (1-P(||\mathbf u||_{\mathbb R^3})\rho \mathbf u^\top (\mathbf u - 
		\tilde{\mathbf u})) \text { } dx. 
	\end{split}
	\]
	We use the Lipschitz solution to write:
	\[
	\begin{split}
		&\rho [\partial_t \tilde{\mathbf u} +  (\nabla_x^\top \tilde{\mathbf u})\mathbf u]^\top
		(\tilde{\mathbf u} - {\mathbf u}) = ... \\  
		& \rho(\tilde{\mathbf u} - {\mathbf u})^\top 
		(\nabla_x^\top \tilde{\mathbf u}) (\tilde{\mathbf u} - {\mathbf u}) + 
		\rho (\partial_t \tilde{\mathbf u}^\top + \tilde{\mathbf u}^\top \nabla_x\tilde{\mathbf u}  )
		(\tilde{\mathbf u} - {\mathbf u}).
	\end{split}
	\]
	For the first term, since $\tilde{\mathbf u} \in \text{Lip}([t_0,t_f] \times \text{cl }D;\mathbb R^3)$, it follows we can select some constant s.t. $||\tilde{\mathbf u}||_{\text{Lip}([t_0,t_f] \times \text{cl } D;\mathbb R^3)} < \mathfrak v$.
	\[
	\begin{split}
		&\int_{D} \rho({\mathbf u} - \tilde{\mathbf u})^\top 
		\nabla_x^\top \tilde{\mathbf u} ({\mathbf u} - \tilde{\mathbf u}) \text { } dx
		\leq 2 \mathfrak v
		\int_{D} \frac{1}{2}\rho ||\mathbf u 
		- \tilde{\mathbf  u} ||_{\mathbb R^3}^2 \text { } dx 
		\leq C(\mathfrak v) \mathbf E[\cdots](t) \text { } dt.
	\end{split}
	\]
	For the second term, we insert:
	\[
	\begin{split}
		\partial_t \tilde{\mathbf u} + (\nabla_x^\top \tilde{\mathbf u}) &\tilde{\mathbf u} 
		= \pi_{\mathbf j}[r \tilde{\mathbf u}] - 
		\tilde{\mathbf u} \pi_{\rho}[r] - \nabla_x(v_r[r] + 
		v_c[r] + U) + \kappa_p \tilde{\mathbf u} (1-P(|| \tilde{\mathbf u} ||_{\mathbb R^3})).
	\end{split}
	\]
	The relative energy inequality becomes:
	\[
	\mathbf E[(\rho,\mathbf j),(r, \mathbf i)](t) \Big |_{t_0}^\tau 
	\leq 
	\int_{t_0}^\tau \tilde{\mathbf R}[(\rho,\mathbf j),(r, \mathbf i)](t) + C \mathbf E[(\rho,\mathbf j),(r, \mathbf i)](t) \text { } dt
	\]
	with:
	
	\[
	\begin{split}
		&\tilde{\mathbf R}[(\rho,\mathbf j),(r, \mathbf i)](\cdot) := 
		\int_{D} \rho  (\mathbf u - \tilde{\mathbf u})^\top
		\nabla_x v_r [\rho-r]  + \rho  (\mathbf u - \tilde{\mathbf u})^\top
		\nabla_x v_c [\rho-r] \text { } dx + ... \\
		&\int_{D} ( \pi_{\mathbf j}[
		\rho {\mathbf u} - r \tilde {\mathbf u}] -  {\mathbf u} 
		\pi_{\rho}[\rho] + \tilde {\mathbf u} \pi_{\rho}[r])^\top \rho({\mathbf u} -  \tilde{\mathbf u} ) \text { } dx + ... \\
		& \int_{D }\kappa_p [(1-P(||\mathbf u||_{\mathbb R^3}) \mathbf u - (1-P(||\tilde {\mathbf u}||_{\mathbb R^3}) \tilde {\mathbf u})^\top \rho (\mathbf u - 
		\tilde{\mathbf u}))] \text { } dx. 
	\end{split}
	\]
	The function $\Phi_p(v)= v(1-P(||v||_{\mathbb R^3}))$ is 
	globally Lipschitz. Denote the Lipschitz constant by $L(\Phi_p)$. So, we obtain a bound on the 
	final integral using this Lipschitz constant:
	\[
	\begin{split}
		&\int_{D }\kappa_p [(1-P(||\mathbf u||_{\mathbb R^3})\mathbf u^\top -
		(1-P(||\tilde{\mathbf u}||_{\mathbb R^3})
		\tilde{\mathbf u}^\top
		]\rho(\mathbf u - 
		\tilde{\mathbf u}) \text { }  dx 
		\leq ... \\ 
		&2\kappa_p L(\Phi_p) \int_{D } \frac{1}{2} \rho ||\mathbf u - 
		\tilde{\mathbf u}||_{\mathbb R^3}^2 dx = C(\kappa_p,P) \mathbf E[\cdots](t) \text { } dt.
	\end{split}
	\]
	Define $\mathfrak B : L^2(t_0,t_f; L^2 (\text{cl }D;\mathbb R \times \mathbb R^3)) \rightarrow L^2(t_0,t_f; L^2 (\text{cl }D;\mathbb R^3))$ as:
	\[
	\mathfrak B[\rho,\mathbf u]:= 
	\pi_{\mathbf j}[\rho {\mathbf u}] -  {\mathbf u} 
	\pi_{\rho}[\rho].
	\]
	By construction of $\pi_{\mathbf j}$ and via Fourier multipliers, for $(\rho, \mathbf j)$ which also satisfy the definition of a bounded-energy weak solution in (\ref{weak1}) - (\ref{weak3}), and for $\mathbf u = \rho^{-1} \mathbf j$:
	\[
	\pi_{\mathbf j}[\rho {\mathbf u}] \leq \frac{k_a}{\lambda_a} ||\rho \mathbf u||_{L^2(t_0,t_f;L^2(D;\mathbb R^3))} = C(\bar \rho, k_a, \lambda_a)||\mathbf u||_{L^2(t_0,t_f;L^2(D;\mathbb R^3))}
	\]
	and:
	\[
	{\mathbf u} 
	\pi_{\rho}[\rho] \leq C(\bar \rho, k_a, \lambda_a)||\mathbf u||_{L^2(t_0,t_f;L^2(D;\mathbb R^3))}.
	\]
	
	Let $Y$ be the set of functions which satisfy Definition \ref{def_3_8}. $\mathfrak B$ Since $\mathfrak B$ is bilinear over $L^2(t_0,t_f; L^2(D;\mathbb R \times \mathbb R^3))$ and bounded over $Y \subset \text{dom } \mathfrak B \subset L^2(t_0,t_f; L^2(D;\mathbb R \times \mathbb R^3)) $, it is rather trivial to show that in any ball $B(0, \epsilon) \subset Y $,
	the functional $\mathfrak B \Big |_Y$ (the restriction of $\mathfrak B$ to $Y$) is s.t.:
	\[
	\begin{split}
		&||\mathfrak B \Big |_Y[\rho -r, \mathbf u - \tilde{\mathbf u}]||_{L^2(t_0,t_f; L^2 (\text{cl }D;\mathbb R^3))} \leq ... \\ 
		&L(\epsilon;\tilde \rho, \bar \rho, k_a, \lambda_a) \Big (||\rho - r||_{L^2(t_0,t_f; L^2 (\text{cl }D;\mathbb R))} + ||\mathbf u - \tilde{\mathbf u}||_{L^2(t_0,t_f; L^2 (\text{cl }D;\mathbb R^3))} \Big ).
	\end{split}
	\]
	
	Now, for the alignment terms:
	\[
	\begin{split}
		&\int_{D} ( \pi_{\mathbf j}[
		\rho {\mathbf u} - r \tilde {\mathbf u}] -  {\mathbf u} 
		\pi_{\rho}[\rho] + \tilde {\mathbf u} \pi_{\rho}[r])^\top \rho({\mathbf u} -  \tilde{\mathbf u} ) \text { } dx = ... \\
		& \int_{D} B[\rho - r, \mathbf u - \tilde{\mathbf u}]^\top \rho (\mathbf u - \tilde{\mathbf u}) \text { } dx.
	\end{split}
	\]
	We can conclude using the previous result and the earlier equivalence of norms we showed earlier that:
	\[
	\begin{split}
		\int_{D} ( \pi_{\mathbf j}[
		\rho {\mathbf u} - r \tilde {\mathbf u}] - & {\mathbf u} 
		\pi_{\rho}[\rho] + \tilde {\mathbf u} \pi_{\rho}[r])^\top \rho({\mathbf u} -  \tilde{\mathbf u} ) \text { } dx \leq ... \\
		L(\epsilon;\tilde \rho, \bar \rho, k_a, \lambda_a)&\Big (||\rho - r||_{L^2(t_0,t_f;L^2(D;\mathbb R))}^2 + ||\mathbf u - \tilde{\mathbf u}||_{L^2(t_0,t_f;L^2(D;\mathbb R^3))}^2 \Big ) \leq ... \\ 
		&C(\cdots) \mathbf E[(\rho,\mathbf j),(r, \mathbf i)](t)
	\end{split}
	\]
	where $\epsilon := \max \{||(\rho, \mathbf u)||_{L^2(t_0,t_f; L^2 (\text{cl }D;\mathbb R \times \mathbb R^3))}, ||(r, \tilde{\mathbf u})||_{L^2(t_0,t_f; L^2 (\text{cl }D;\mathbb R \times \mathbb R^3))}\}$. Finally, for the first term, via Young's convolution inequality and equivalence of norms:
	\[
	\begin{split}
		&\int_{D} \rho  (\mathbf u - \tilde{\mathbf u})^\top
		\nabla_x v_r [\rho-r] \text{ } dx  \leq ... \\ 
		&\bar \rho ||\mathbf u - \tilde{ \mathbf u}||_{L^2(D;\mathbb R^3)} ||\nabla_x G(\cdot, k_r, \lambda_r)||_{L^1(D;\mathbb R^3)}||\rho - r||_{L^2(D;\mathbb R)} \leq ... \\
		&\frac{C(\bar \rho, k_r,\lambda_r)}{2} \Big (||\mathbf u - \tilde{ \mathbf u}||_{L^2(D;\mathbb R^3)}^2 +  ||\rho - r||_{L^2(D;\mathbb R)}^2 \Big) \leq ... \\ 
		& C(\tilde \rho, \bar \rho, k_r, \lambda_r) \mathbf  E[(\rho, \mathbf j),(r, \mathbf i)](t).
	\end{split}
	\]
	We perform a virtually identical procedure on the terms containing $\nabla_x v_c[\rho - r]$.
	We have successfully bounded all terms of the remainder by the relative energy. Finally, we have arrived at:
	\[
	\mathbf E[(\rho, \mathbf j), (r, \mathbf i)](t) \Big |_{t_0}^\tau 
	\leq C(\cdots) \int_{t_0}^\tau \mathbf E[(\rho, \mathbf j), (r, \mathbf i)](t) \text{ } dt,
	\]
	with positive constant $C$ independent of $\mathbf E$.
	The relative energy is obviously non-negative and initially zero. Since the inequality we have obtained holds for a.e. $t \in [t_0,t_f]$, 
	by the Gr\"onwall-Bellman inequality, we conclude that $\mathbf E[\cdots](t) = 0$ a.e. on $[t_0,t_f]$. Since $\mathbf E$ is comprised of norms, the claims of Theorem \ref{prop3} follow. This completes the proof.
\end{proof}

\section{Generalizations and Corollaries}
\label{section4}
In this section, we present two extensions to our model. First, we give a model of a swarm subject to an Euler alignment system which follows some controlled leader-agents whose position-velocity dynamics satisfy controlled ODEs. This gives a sparsely controlled system. We use our previous results to sketch a proof of existence of bounded-energy weak solutions and weak-strong uniqueness for this new model. Then, we present some results where the Yukawa potential of our original model has been replaced with Bessel potentials of fractional order $s>2$, and similarly use our previous results and Bessel potential spaces to sketch proofs for existence of bounded-energy weak solutions and weak-strong uniqueness.
\subsection{Sparse Control of the Density by Leaders}
Consider ODEs for positions and velocities of leaders $(\xi_i,\upsilon_i):[t_0,t_f] \rightarrow \mathbb R^3 \times \mathbb R^3$, $i = 1, 2, ..., k$:
\begin{equation}
	\label{ode_leader}
	\begin{cases}
		\dot \xi_i(t)  = \upsilon_i(t) \\
		\dot \upsilon_i(t) = \mathbf F(\xi_i(t),\upsilon_i(t), f_i(t);\rho(t,\cdot))
	\end{cases}
\end{equation}
where $f_i \in \text{PWC}([t_0,t_f];\mathbb R^3)$ is a control for the $i$-th leader, and $\rho \in C([t_0,t_f] \times D;\mathbb R)$ a probability density in space at each time $t \in [t_0,t_f]$.

Suppose that $(\upsilon,\mathbf F)$ leaves invariant the set $(\tilde D \subset D) \times 
B(0,s)$, $s>0$ sufficiently large, with $\tilde D$ a nonempty, closed, connected, smooth subset, i.e.
$$
\upsilon^\top \nu_\xi + \mathbf F^\top(\xi,\upsilon, f(t), \rho(t,\cdot)) \nu_\upsilon \leq 0 \text{ on } \partial (\tilde D \times B(0,s)) \text{ } \forall t \in [t_0,t_f]
$$
for $f, \rho$ arbitrary.
Then, if each $(\xi_i(t_0),\upsilon_i(t_0))^\top = (\xi_i^0, \upsilon_i^0)^\top \in \tilde D \times B(0,s)$, there is a unique classical solution to the ODEs bound to $\tilde D \times B(0,s)$ for each $i$. Now, couple the given ODEs to the Euler alignment system 
through the potential $U$:
\begin{equation}
	\begin{cases}\label{euler_leader}
		\partial_t \rho + \nabla_x \cdot \mathbf j = 0 \text{ in } (t_0,t_f] \times D \\
		\partial_t \mathbf j + \nabla_x \cdot [\rho^{-1}\mathbf j \mathbf j^\top] = 
		\mathbf S[\rho, \mathbf j]  \text{ in } (t_0,t_f] \times D \\
		\Lambda_a \mathbf \pi_\rho[\rho] = \mathbb I_D \rho \text{ in } \mathbb R^3\\
		\Lambda_a \mathbf \pi_{\mathbf j}[\mathbf j]= \mathbb I_D \mathbf j \text{ in } \mathbb R^3 \\
		\Lambda_c v_c[\rho] = \mathbb I_D \rho \text{ in } \mathbb R^3  \\
		\Lambda_r v_r[\rho] = \mathbb I_D \rho \text{ in } \mathbb R^3 \\
		\mathbf S[\rho, \mathbf j] := 
		\rho \pi_{\mathbf j}[\mathbf j] - \mathbf j \pi_\rho[\rho] - \rho \nabla_x[\mathbf 1_2^\top \mathbf V[\rho] + U(\cdot; \{ \xi_i(\cdot)\}_{i=1}^k)] 
		+ ... \\
		\kappa_p \mathbf j (1-P(\rho^{-1}||\mathbf j||_{\mathbb R^3} )) \\
		\mathbf j^\top \mathbf \nu = 0 \text{ on } \partial D \times [t_0,t_f] \\ 
		(\rho(t_0,\cdot), \mathbf j(t_0,\cdot)) = (\rho_0, \mathbf j_0) \text{ in } D.
	\end{cases}
\end{equation}
We have the following result: 
\begin{theorem}
	Consider the coupled ODE-PDE system (\ref{ode_leader}, \ref{euler_leader}) and suppose that $U(\cdot;\{ \xi_i\}_{i=1}^k)$ is additionally smooth and bounded w.r.t. $\{ \xi_i\}_{i=1}^k$, and that:
	\[
	\sup_{x \in D, \{ \xi_i\} \in D^k} ||\nabla_x U(x;\{ \xi_i\}_{i=1}^k)||_{\mathbb R^3} 
	< \bar U_g,
	\] 
	$\bar U_g$ a positive constant.
	Then, taking the other hypotheses of Theorems \ref{prop1}, \ref{prop2}, \ref{prop3}, their conclusions hold for this coupled ODE-PDE system.
\end{theorem}
\begin{proof}(Sketch.) Let $\mathbf x_i[\mathbf v] = \xi_i(\cdot)$. $\mathbf x_i$ is constant w.r.t. $\mathbf v$, and is obviously $Q$-continuous, as is $U(\cdot;\{\mathbf x_i[\mathbf v] \}_{i=1}^k)$. Now, construct using the new structure of $U$ the operator $\mathbf M[\cdot]$ using the Dirichlet problem for the given elliptic system (\ref{elliptic}). The operator $\mathbf M$ is still $Q$-continuous. So, the result of Theorem \ref{prop1} follows. The proof of Theorem \ref{prop2} need not be modified significantly, nor does the proof of Theorem \ref{prop3}.
\end{proof}
Another result of the flexibility we have in these methods w.r.t. constructing momentum fields for densities of our choosing is the following result on reachability for the density:
\begin{corollary}
	Under the assumptions of the previous result, there are for every 
	$r \in H^2(D;\mathbb R) \cap \text{Lip}(\text{cl }D;\mathbb R) \text{ s.t. }r > 0 \text{ in } D$ controls 
	$f_i \in PWC([t_0,t_f];\mathbb R^3)$ s.t. $\rho(t_f,\cdot) = r$.
\end{corollary}
\begin{proof}(Sketch.) Pick some controls $f_i$. When we select the densities $$\rho \in 
	C^3([t_0,t_f];H^2(D;\mathbb R)) \cap \text{Lip}([t_0,t_f] \times \text{cl }D;\mathbb R), \rho > 0, \rho(t_0,\cdot) = \rho_0$$ 
	in Theorem \ref{prop1}, simply include the additional 
	constraint $\rho(t_f,\cdot) = r$. Then, perform the Helmholtz decomposition on $\mathbf j$, and apply the previous results.
\end{proof}
\subsection{Bessel Potentials}
We also have the following generalization to Bessel potentials:
\begin{theorem}
	\label{bessel_pot}
	Suppose $\Lambda_\alpha$ is of form:
	\[
	\Lambda_\alpha = \frac{1}{k_\alpha}
	(I - \nabla_x^2)^{\lambda_\alpha/2} \Leftrightarrow \mathcal L_\alpha = k_\alpha(I - \nabla_x^2)^{-\lambda_\alpha/2}
	\]
	with $\lambda_\alpha \geq 2$.
	If the other assumptions of Theorems \ref{prop1}, \ref{prop2}, \ref{prop3} hold, then their conclusions hold for this choice of operator.
\end{theorem}
\begin{proof} (Sketch.) The sketch of the proof describes how to modify the arguments of Theorems \ref{prop1}, \ref{prop2}, and \ref{prop3} to suit the new choice of non-local terms.
	
	For $\lambda_\alpha = 2$, the non-local operator given above is the same type as the original choice, so we focus on $\lambda_\alpha > 2$. To start the arguments, as we've done before, pick arbitrary 
	\[
	\rho \in C^3([t_0,t_f];H^2(D;\mathbb R))\cap \text{Lip}([t_0,t_f] \times \text{cl }D;\mathbb R), \rho > 0, \rho(t_0,\cdot) = \rho_0.
	\]
	Define the Bessel potential spaces \cite{adams2012function} for $s > 0, 1 < p < \infty$:
	\[
	B^{s,p}(\mathbb R^3;\mathbb R):= \{f \in \mathcal S'(\mathbb R^3;\mathbb R ): \mathfrak F^{-1} \Big [(1 + \xi^\top \xi )^{s/2} \mathfrak F[f] \Big ] \in L^p(\mathbb R^3;\mathbb R) \}
	\]
	with norm:
	\[
	||f||_{B^{s,p}(\mathbb R^3;\mathbb R)} := ||\mathfrak F^{-1} \Big [(1 + \xi^\top \xi )^{s/2} \mathfrak F [f] \Big ]||_{L^p(\mathbb R^3;\mathbb R)}.
	\]
	We have used $\mathfrak F [f]$ to be the Fourier transform of $f$, and $\mathfrak F^{-1}[f]$ to be the inverse Fourier transform of $f$. See \cite{evans} to review their properties.
	These spaces $B^{s,p}(\mathbb R^3;\mathbb R)$ can be seen as interpolation spaces between the standard Sobolev spaces $W^{\text{ceil}(s),p}(\mathbb R^3;\mathbb R)$ and 
	$W^{\text{floor}(s),p}(\mathbb R^3;\mathbb R)$ for non-integer $s$, and if $s$ is an integer, $B^{s,p}(\mathbb R^3;\mathbb R)$ coincides with a standard Sobolev space under an equivalent norm \cite{adams2012function, chandler2017sobolev}.
	Consider this definition to extend in the usual way to vector fields. For bounded, open, smooth sets $D \subset \mathbb R^3$, we define using restrictions \cite{chandler2017sobolev}:
	\[
	B^{s,p}(D;\mathbb R):= \{f \in \mathcal S'(D;\mathbb R) : 
	\text{ } \exists \text{ } g \in B^{s,p}(\mathbb R^3;\mathbb R)\text{ s.t. } f  = g \big |_{D} \}, 
	\]
	with norm
	\[
	||f||_{B^{s,p}(D;\mathbb R)}:= \inf_{g \in B^{s,p}(\mathbb R^3;\mathbb R), g |_{D} = f} ||g||_{B^{s,p}(\mathbb R^3;\mathbb R)}.
	\]
	In particular, if $v \in L^\infty(D;\mathbb R)$, 
	\[
	\mathcal L_{\alpha} \mathbb I_D v = k_\alpha G_{\lambda_{\alpha}}(\cdot)*\mathbb I_D v \in B^{\lambda_{\alpha}, p}(D;\mathbb R) \text{ for each } 1 < p < \infty,
	\]
	where $G_{\lambda_{\alpha}}$ is the Bessel  kernel \cite{aronszajn1961theory}:
	\[
	G_{\lambda_\alpha}(x):=\frac{1}{2^{\frac{3+\lambda_\alpha - 2}{2} 
		}\pi^{3/2} \Gamma(\frac{\lambda_\alpha}{2})}K_{\frac{3-\lambda_\alpha}{2}}(||x||_{\mathbb R^3})||x||_{\mathbb R^3}^{\frac{\lambda_\alpha-3}{2}}
	\]  
	of order $\lambda_\alpha > 0$, and $K_\nu$ is the modified Bessel function of second kind.
	We have already selected $\rho$, so consider $\mathbf j \in X$, $\mathbf j(t_0,\cdot) = \mathbf j_0$, not necessarily a weak solution to (\ref{euler2}) with the new interactions. $(\rho, \mathbf j)^\top$ is regular enough so that for $\lambda_a, \lambda_c, \lambda_r$ there are embeddings (after first embedding into 
	$W^{\text{floor}(\lambda_\alpha),p}(\cdots)$):
	\[
	\pi_{\mathbf j}[\mathbf j](t,\cdot) \in B^{\lambda_a,p}(D;\mathbb R^3) \subset 
	C^{1,a}(D;\mathbb R^3) \text{ if } \mathbf j \in X, t \in (t_0,t_f)
	\]
	and
	\begin{equation*}
		\begin{split}
			\pi_{\rho}[\rho](t,\cdot), v_r[\rho]&(t,\cdot), v_c[\rho](t,\cdot)  \in ... \\  & B^{\min\{\lambda_a, \lambda_c, \lambda_r\}+2,2}(D;\mathbb R^2) \subset 
			C^{2,1/2}(D;\mathbb R^2), t \in (t_0,t_f)
		\end{split}
	\end{equation*}
	for some $2 < p < \infty, 0 \leq a < 1$. Simply apply the arguments of Theorems \ref{prop1}, \ref{prop2}, \ref{prop3}  modified slightly for the new potentials and their conclusions follow.
\end{proof}
\section{Conclusion}
\label{section5}
In this work, we studied a (formal) hydrodynamic limit of a generalized Cucker-Smale model. The novelties of these problems compared to \cite{weak_sols}, for example, are:
\begin{enumerate}
	\item A bounded domain as opposed to the torus;
	\item Singular alignment, repulsion, and cohesive potentials, in particular the Yuka\-wa potential, a type of Bessel potential;
	\item Addition of a confinement potential;
	\item Incorporation of leaders and sparse control; and
	\item A generalization to Bessel potentials of arbitrary order.
\end{enumerate}
In particular, unique Lipschitz densities are admitted by the Euler alignment system, and solutions with Lipschitz momenta are unique in the class of bounded-energy weak solutions. Assuming Lipschitz solutions exist for arbitrarily long times, we may use these conclusions \textit{a priori} in system identification problems and sparse (optimal) control problems for biological and engineered swarms, respectively.

\bibliographystyle{siam}  
\bibliography{references.bib}
\end{document}